\documentclass{amsart}

\usepackage{amsmath}
\usepackage{amsfonts}
\usepackage{amssymb}
\usepackage{comment}
\usepackage{epsfig}
\usepackage{psfrag}
\usepackage{mathrsfs}
\usepackage{amscd}
\usepackage[all]{xy}
\usepackage{rotating}
\usepackage{lscape}
\usepackage{amsbsy}
\usepackage{verbatim}
\usepackage{moreverb}

\usepackage{color}

\newtheorem{theorem}{Theorem}[section]
\newtheorem{prop}[theorem]{Proposition}
\newtheorem{lemma}[theorem]{Lemma}
\newtheorem{cor}[theorem]{Corollary}

\newtheorem{expectation}[theorem]{Expectation}

\newtheorem{definition}[theorem]{Definition}

\theoremstyle{remark}
\newtheorem{remark}[theorem]{Remark}
\newtheorem{example}[theorem]{Example}
\newtheorem{q}[theorem]{Question}

\DeclareMathOperator{\Aut}{Aut}

\DeclareMathOperator{\colim}{colim}

\DeclareMathOperator{\Hom}{Hom}
\DeclareMathOperator{\Fun}{Fun}

\DeclareMathOperator{\Tot}{Tot}

\DeclareMathOperator{\Comod}{Comod}
\DeclareMathOperator{\Mod}{Mod}

\DeclareMathOperator{\qc}{\mathcal Q}

\def\1bord{1\mathrm{Bord}}
\def\2bord{2\mathrm{Bord}}
\def\3bord{3\mathrm{Bord}}

\DeclareMathOperator{\hh}{HH}

\newcommand{\ra}{\rightarrow}

\newcommand{\inv}{{}^{-1}}

\def\cA{\mathcal A}\def\cB{\mathcal B}\def\cC{\mathcal C}\def\cD{\mathcal D}
\def\cE{\mathcal E}\def\cF{\mathcal F}\def\cH{\mathcal H}
\def\cI{\mathcal I}\def\cK{\mathcal K}\def\cL{\mathcal L}
\def\cM{\mathcal M}\def\cO{\mathcal O}\def\cP{\mathcal P}
\def\cS{\mathcal S}

\def\cZ{\mathcal Z}

\def\CC{\mathbb C}\def\DD{\mathbb D}

\def\SS{\mathbb S}

\def\ZZ{\mathbb Z}

\def\ot{\otimes}

\def\ot{\otimes}
\def\XYX{X \times_Y X}
\def\fM{\mathfrak M}
\def\fz{\mathfrak z}

\def\ftr{\mathfrak {tr}}

\def\wt{\widetilde}
\def\D{{\mathcal D}}

\newcommand{\on}{\operatorname}
\newcommand{\Vect}{\on{Vect}}
\newcommand{\Rep}{\on{Rep}}

\newcommand{\Ind}{\on{Ind}}
\newcommand{\Perf}{\on{Perf}}

\newcommand{\ti}{\times}

\newcommand{\Id}{{\rm id}}
\newcommand{\intHom}{{\mathcal Hom}}
\newcommand{\intEnd}{{\mathcal End}}

\newcommand{\mc}{\mathcal}

\def\oo{\infty}

\def\Tr{\mathcal Tr}

\newcommand{\bs}{\backslash}

\newcommand{\Idem}{{st}}
\newcommand\Cat{\on{Cat}}

\newcommand\Alg{\mathcal Alg}

\def\sn{\section}
\def\ssn{\subsection}
\def\sssn{\subsubsection}

\def\End{\on{End}}
\def\intEnd{{\mathcal End}}

\def\fg{\mathfrak g}

\def\fL{\mathfrak L}
\def\fM{\mathfrak M}
\def\fN{\mathfrak N}

\def\Tr{\on{Tr}}
\def\Z{\on{Z}}

\def\chsh{{Ch}}

\def\St{St}

\def\Dcoh{\mathcal D_{coh}}

\def\cP{\mathcal P}

\def\Ind{\on{Ind}}

\def\catop{\diamondsuit}

\def\fh{\mathfrak h}

\def\adjquot{/_{\hspace{-0.2em}ad}\hspace{0.1em}}

\newcommand{\Xtil}{\tilde X}

\newcommand{\cHtil}{\widetilde \cH}

\begin{document}

\title{The character theory of a complex group}
\author{David Ben-Zvi and David Nadler}

\address{Department of Mathematics\\University of Texas\\Austin, TX 78712-0257}
\email{benzvi@math.utexas.edu}
 \address{Department of Mathematics\\University
  of California\\Berkeley, CA 94720-3840}
\email{nadler@math.berkeley.edu}

\maketitle

\newcommand{\GoG}{\displaystyle{\frac{G}{G}}}
\newcommand{\BoB}{\displaystyle{\frac{B}{B}}}
\newcommand{\BGB}{B\bs G/B}
\newcommand{\Gv}{G^{\vee}}

\newcommand{\findim}{\mathit{fd}}





\begin{abstract}
We apply the ideas of derived algebraic geometry and topological field
theory to the representation theory of reductive groups.  Our focus is the Hecke category
of Borel-equivariant $\D$-modules on the flag variety of a complex reductive group~$G$
(equivalently, the category of Harish Chandra bimodules of trivial central character)
and its monodromic variant. The Hecke category
is a categorified analogue of the finite Hecke algebra, which is a finite-dimensional
semi-simple symmetric Frobenius algebra. We establish 
parallel properties of the Hecke category, showing it is a two-dualizable Calabi-Yau monoidal
category, so that
in particular, its monoidal (Drinfeld) center and
trace 
coincide. We calculate that they are identified through the Springer
correspondence with Lusztig's unipotent character
sheaves. 
It follows that Hecke module categories, such as categories
of Lie algebra representations and Harish Chandra modules for $G$ and
its real forms, have characters which are themselves character sheaves.
Furthermore, the Koszul duality for Hecke categories 
provides 
a Langlands duality for unipotent character sheaves.  This can be viewed as part of a
dimensionally reduced version of the geometric Langlands
correspondence, or as $S$-duality for a maximally supersymmetric gauge
theory in three dimensions.



\end{abstract}

\tableofcontents



\section{Introduction}


Geometric approaches to the representation theory of Lie groups $G$ are
intimately linked with the study of $\D$-modules.  For example, Harish
Chandra established his celebrated regularity results for
distributional characters of admissible representations of real reductive
groups by showing that they satisfy an adjoint-equivariant regular
holonomic system of differential equations.  The resulting
$\D$-module is the fundamental example of a {character sheaf in the
  sense of Lusztig} and also is the basic object in Springer theory.

In another direction, Harish Chandra introduced $(\fg, K)$-modules to
capture the underlying algebraic structure of representations of real
reductive groups while avoiding the analytic intricacies involving 
function spaces. In turn, Beilinson-Bernstein opened the study of $(\fg, K)$-modules to the powerful sheaf-theoretic techniques
of algebraic geometry by establishing a  localization theorem 
  identifying them with
$K$-equivariant $\D$-modules on the flag variety $G/B$. (More precisely, representations with strictly trivial infinitesimal character
are identified with $\D$-modules on $G/B$, while representations with  generalized trivial infinitesimal character correspond to unipotent monodromic $\D$-modules on $G/N$.) 
Other fundamental categories in  representation
theory of reductive groups, such as the highest weight representations
of Category~$\mathcal O$, are similarly equivalent to categories of (possibly monodromic)
$\D$-modules on flag varieties and related spaces.

Going one step further, the natural symmetries of representations,
such as the intertwining operators for principal series
representations, correspond to integral transforms acting on 
$\D$-modules.  In the most basic instance of this, the collection of
all integral transforms acting on 
$\D$-modules on the flag variety $G/B$ forms the 
Hecke category
of Borel biequivariant 
$\D$-modules on $G$. 
Composition of integral transforms, geometrically described by
convolution of integral kernels, equips the Hecke category 
with a monoidal structure which categorifies the classical Hecke
algebra associated to the Weyl group of $G$. Moreover, other fundamental categories of representations, such as categories of $(\fg,K)$-modules and Category
$\mathcal O$, are naturally
module categories for the Hecke category. 

Our main results relate these two classes of
categories: on the one hand, the Harish Chandra system and other
character sheaves, and on the other hand, the Hecke category and its module categories. We first establish 
that the Hecke category, along with its monodromic variant, 
satisfies categorified analogues of the fundamental properties
of the finite Hecke algebra.
Namely, the Hecke category is a 
two-dualizable Calabi-Yau monoidal category
as the
finite Hecke calgebra is a finite-dimensional semi-simple symmetric Frobenius algebra. 
We then
calculate that the monoidal (Drinfeld) center and trace of the Hecke category, as well as those of its monodromic variant, 
are identified with 
Lusztig's unipotent character sheaves. 
It follows that dualizable Hecke module categories have characters which are character sheaves, 
a categorification of Harish Chandra's
theory of characters. (It is possible to go further and apply this perspective to give a geometric construction of the Harish Chandra characters themselves.)

The technical setting for our work is that of homotopical algebra and derived algebraic geometry.
We apply the powerful tools of this rapidly developing subject to show that 
operations on the ``function spaces" of algebraic analysis, namely categories
of $\D$-modules, are representable by integral transforms and to analyze the algebraic
structure of these operations.

Our main results are perhaps best understood in
the framework of extended topological field theory, using Lurie's
proof of the Cobordism Hypothesis\;\cite{jacob TFT}.  We explain that
there is a topological field theory (TFT), which we call the 
{\em  character theory}, 
that organizes much of
the geometric representation theory associated to  $G$.  (In  work in progress with Sam Gunningham, 
we plan to construct a family of TFTs parametrized by central character -- the ``moduli of vacua" of the theory -- of which the unipotent
 character theory discussed here is the fiber
over the trivial parameter.) The character theory can be viewed as an extended
3-dimensional TFT, defined on $0,1$ and
$2$-dimensional manifolds, or alternatively, as a categorified
2-dimensional TFT.  It assigns to a point the
Hecke category itself (viewed as an object in a higher Morita category of monoidal categories), or equivalently, its 2-category of modules.
It assigns to a circle the category of unipotent character sheaves, which thus carries the rich operadic structure of TFT (such as a
ribbon tensor structure). Moreover, the Koszul duality of Hecke categories, exchanging the equivariant and monodromic variants, implies that the  character theory
satisfies a Langlands duality.
This can
be viewed as a dimensionally reduced form of the geometric Langlands
conjecture, or as $S$-duality for a maximally supersymmetric gauge
theory in three dimensions.

\medskip

The remainder of the Introduction is organized as follows. After a brief summary of notation and conventions, in Section~\ref{rep
  intro}, we state our main results, summarized in Theorem~\ref{intro Hecke thm},  from the perspective of
representation theory; in Section~\ref{tools}, we discuss the
underlying techniques of homotopical algebra for $\D$-modules; and in
Section~\ref{TFT}, we explain how our results fit into the framework of
topological field theory. Finally, in Section~\ref{families}, we propose a natural extension of the results of this paper
to the context of twisted $\D$-modules.
%
%
%
%
%
%
%
%

\medskip
{\em  Notation and conventions.}
Throughout the paper, $G$ will denote a  complex reductive group
with Lie algebra $\fg$. We choose a 
 Borel subgroup $B\subset G$ with unipotent radical $N\subset B$. We denote by  $H=B/N$
 the universal 
 Cartan torus with Lie algebra $\fh$.

  We also adopt the following idiosyncratic notation: we
 write $G\adjquot G$ to denote the adjoint quotient, or
 in other words, the quotient of $G$ by itself acting by
 conjugation. More generally, for a subgroup $H\subset G$, we will
 write $G\adjquot H$ to denote the quotient of $G$ by the subgroup $H$
 acting by conjugation. In this way, we can write $G\adjquot B$ and not
 confuse the quotient with the flag variety $G/B$.

Since our methods and results will involve homotopical algebra
and topological field theory, we find it most natural to work in the
context of differential graded (dg) and $\infty$-categories. By a dg category, we
will  always mean a pre-triangulated $\CC$-linear dg category. Since the homotopy theory
of such categories is identified with that of stable $\CC$-linear $\infty$-categories, we will use the
terms dg category and stable $\infty$-category interchangeably. Most of our constructions will take place within
the symmetric monoidal $\infty$-category $\St_\CC$ of stable presentable $\CC$-linear $\oo$-categories
with morphisms continuous (colimit-preserving)
functors. We will usually abuse terminology and refer to its objects as categories.
For foundations on $\oo$-categories (in the sense of
Joyal~\cite{Joyal}), we refer to the comprehensive work of Lurie
\cite{topos,HA}.  See Sections~\ref{infinity},
 \ref{sect ind-categories}, \ref{monoidal infinity} for a summary of
what we will need of this theory.

By a scheme �$X$, we will always mean a quasicompact, separated derived
scheme of finite type over $\CC$. 
By a $\D$-module on a smooth scheme $X$, we will always mean a complex of $\D$-modules, and all functors
will be derived. Thus the category of $\D$-modules $\D(X)$ means the dg category of complexes of $\D$-modules
which we can regard as an object of $\St_\CC$.

%
%


\ssn{Homotopical algebra of Hecke categories}\label{rep
  intro} Here we state our main results, 
  summarized in Theorem~\ref{intro Hecke thm},
   from the perspective of
representation theory.
  We first recall 
  two categories of $\D$-modules of fundamental
interest in  representation theory and of primary focus in this paper.
%
%


\sssn{Hecke categories}
%

\begin{definition}
1) The Hecke category  is the dg category $$\cH_G = \D_G(G/B\times G/B)\simeq \D(B\bs G/B)$$
of all $B$-biequivariant $\D$-modules on $G$.

2) The monodromic Hecke category  is the full dg subcategory
 $$\cHtil_G \subset \D_G(G/N\times G/N)\simeq \D(N\bs G/N)$$
generated (under colimits) by pullbacks from $\cH_G$.
\end{definition}
%

Both Hecke categories carry natural monoidal structures given by convolution. 
Thus they form algebra objects in the symmetric monoidal $\infty$-category $\St_\CC$.

\sssn{Character sheaves} 

The traditional definition \cite{character 1} of character sheaf (see also~\cite{laumon} for a review and \cite{ginzburg,MV,grojnowski} for geometric
approaches to the theory)
is based on the horocycle
correspondence
$$
\xymatrix{
{G}\adjquot{G} & \ar[l]_-{p} {G}\adjquot{B} \ar[r]^-{\delta} & B\bs G/B.
}
$$ The map $p$ is the natural projection with fibers isomorphic to the flag
variety $G/B$, and the map $\delta$ is the natural projection with fibers
isomorphic to $B$.  Pulling back and pushing forward $\D$-modules
gives a functor called the Harish Chandra transform (in the
terminology of \cite{ginzburg})
$$
\xymatrix{
F=p_*\delta^!:\cH_G=\D(\BGB)\ar[r] &  \D(G\adjquot G)
}$$

A traditional unipotent character sheaf is a (shift of a) $G$-equivariant $\D$-module 
on $G$ that is a simple constituent of 
a $\D$-module
obtained by applying the transform $F$
to a simple $\D$-module (in the heart of the standard $t$-structure on $\cH_G$).
All unipotent character sheaves can be characterized geometrically as the simple 
adjoint-equivariant $\D$-modules  
on $G$   with 
singular support in the
nilpotent cone and unipotent central character. 
Lusztig showed that the collection of all character sheaves provides 
a construction of the characters of finite groups of Lie type, and
gave a detailed classification in relation with the structure of the finite Hecke algebra.

With the above in mind, we make the following definition in the homotopical setting.

\begin{definition}
The dg category $\chsh_G$ of unipotent character sheaves is
defined to be the full subcategory of $\D(G\adjquot G)$ generated
(under colimits) by the image of the Harish
Chandra transform~$F$.
%
%
\end{definition}

Note that traditional unipotent character sheaves are simple objects of the
heart of  $\chsh_G$ with respect to the standard
$t$-structure.

\begin{example}\label{springer}
The above correspondence 
can be viewed as a collection of Weyl group twisted versions of the Grothendieck-Springer simultaneous resolution.
Namely, if we restrict to the support $$pt/B=B\bs B/B \subset B\bs G/B$$ of the 
monoidal unit of $\cH_{G}$,
then we recover (an equivariant global version of) the Grothendieck-Springer correspondence
$$
\xymatrix{
{G}\adjquot{G} & \ar[l]_-{p} {\wt G}/{G} \ar[r]^-{\delta} & B\bs B/B \simeq pt/B.
}
$$
Here ${\wt G}/G$ denotes the simultaneous resolution
$$
{\wt G}/{G} =\{(g,B')\, |\, g\in B'\}/G \simeq B\adjquot B,
$$
and the map $p$ forgets the flag $B'$, while the map $\delta$ forgets the group element $g$.

The global version of the Springer sheaf 
$$\cS_G = F(\cO_{pt/B}) = p_*\delta^!\cO_{pt/B}$$
is the fundamental example of a character sheaf. (Some might prefer to shift $\cS_G$ so that it lies in the heart of the standard
$t$-structure.)
According to~\cite{HK}, it coincides with Harish Chandra's adjoint-equivariant 
holonomic system of differential equations obtained by setting all of the Casimir operators to zero. 
\end{example}



\subsubsection{Prelude to results: analogy with finite groups}

Our results are perhaps most easily understood in analogy with the
following well-known properties of a {finite} group $\Gamma$,
and specifically its complex group algebra
$\CC[\Gamma]$ and category of finite-dimensional complex modules $\Rep^{\findim}_\CC(\Gamma)$.


\medskip
1) The group algebra $\CC[\Gamma]$ is a finite-dimensional
  semi-simple symmetric Frobenius algebra with nondegenerate functional
$$
\xymatrix{
\tau:\CC[\Gamma]\ar[r] & \CC &
\tau(f)=f(e)/|\Gamma|
}
$$ where $e\in \Gamma$ is the unit.

\medskip
2)
The class functions $\CC [\Gamma]^\Gamma$ are the
  center of the group algebra
$$\xymatrix{
\fz:\CC [\Gamma]^\Gamma  \ar@{^(->}[r] & \CC [\Gamma]
}$$
and in particular form a commutative algebra.
Equivalently, they are the
  endomorphisms of the identity functor of $\Rep^{\findim}_\CC(\Gamma)$, and so act
  universally on any  $M\in \Rep^{\findim}_\CC(\Gamma)$.

\medskip
3)
The class functions $\CC [\Gamma]^\Gamma$ are the target of a universal trace map
$$\xymatrix{
\ftr: \CC [\Gamma]\ar@{->>}[r] & \CC [\Gamma]^\Gamma
}$$
initial among all maps that are equal on $ab, ba\in \CC[\Gamma]$, for
   $a, b\in \CC[\Gamma]$. Equivalently, they are the universal
  recipient of a functorial trace map from endomorphisms of any
   $M\in \Rep^{\findim}_\CC(\Gamma)$ which assigns to the identity  the
  character $\chi_M\in \CC [\Gamma]^\Gamma$. 
  
  The symmetric property of the Frobenius algebra is equivalent
  to the fact that the functional $\tau$ is a trace so admits a  factorization
  $$\xymatrix{
  \tau:\CC[\Gamma] \ar[r]^-{\ftr} & \CC[\Gamma]^\Gamma \ar[r] & \CC}$$


It follows that class functions themselves form a commutative Frobenius algebra,
with multiplication and functional induced by those of $\CC[\Gamma]$ via the
center and trace maps.
%

\subsubsection{Results: Hecke categories and character sheaves}
%
To begin, we work with monoidal categories, by which we mean algebra objects $\cC\in Alg(\St_\CC)$ in the symmetric monoidal $\infty$-category of dg (or stable presentable) categories $\St_\CC$, as categorified analogues of $\CC$-algebras. 
There are natural categorical analogues of center and trace 
$$
\xymatrix{
 \Z(\cC) = \End_{\cC\ot  \cC^{op}}(\cC)  & \Tr(\cC) = \cC\ot_{\cC\ot  \cC^{op}} \cC 
}
$$
The first is naturally braided, i.e., an $\cE_2$-algebra in $\St_\CC$, by the general form of the Deligne conjecture (see~\cite{HA}) and comes with a universal central map
$$
\xymatrix{
\fz: \Z(\cC)\ar[r] &  \cC
}
$$
The second has a natural $S^1$-action, a generalization of the Connes cyclic structure on Hochschild chains, and comes with a universal trace map
$$
\xymatrix{
\ftr: \cC\ar[r] &  \Tr(\cC)
}
$$

We propose the following notion of a semi-rigid monoidal category as a categorified analogue of 
finite-dimensional semi-simple algebra (see Section~\ref{TFT} for further justification of this analogy).

\begin{definition}\label{perfect intro def}
A monoidal category $\cC$ is called a {\em
  semi-rigid category} if $\cC$ is compactly generated, and
all of its compact objects (or equivalently, a collection of compact
  generators) are both left and right dualizable.
\end{definition}

\begin{remark}If in addition the unit of $\cC$ is compact (or equivalently,
its compact and dualizable objects coincide), then $\cC$ is also {\em rigid} in the sense of \cite{DGcat} (note however that~\cite{DGcat} does not require $\cC$ to be compactly generated).
For example, a quasi-compact derived stack $X$
with affine diagonal is perfect in the sense of \cite{BFN} when its category $\qc(X)$
of quasicoherent sheaves is rigid. We have introduced the notion of semi-rigid
category to accommodate our main example, the Hecke category, whose unit is not compact.
\end{remark}

There are natural weak and strong analogues of the notion of symmetric Frobenius algebra for semi-rigid monoidal categories, namely pivotal and Calabi-Yau structures. 

\begin{definition}  
 A {\em pivotal structure} on a semi-rigid monoidal category $\cC$ is a monoidal identification of the 
operations of taking left and right duals.
\end{definition}

A pivotal structure  equips $\cC$ with a trace 
$$
\xymatrix{
\tau:\Tr(\cC)\ar[r] & \Vect
}
$$
factoring the functional given on compact objects 
(and all objects if $1_\cC$ is compact) by the pairing
$$
\xymatrix{
\Hom_{\cC}(1_\cC, -): \cC\ar[r] & \Vect
}
$$
We show in Section~\ref{rigid categories} that a pivotal structure on a semi-rigid category induces a canonical identification of
its monoidal center and trace. Pivotal structures however have the disadvantage that they are not manifestly Morita invariant notions.


\begin{definition}\label{CY def}
A {\em Calabi-Yau structure} on a monoidal category $\cC$ is a trace map
$$
\xymatrix{
\Tr(\cC) \ar[r] &  \Vect
  }
  $$
 that is $S^1$-invariant and so that the natural composition
 $$
\xymatrix{
\cC\otimes \cC\ar[r] & \Tr(\cC) \ar[r] &  \Vect
  }
  $$
   is the evaluation of a self-duality of $\cC$.
  \end{definition} 

The notion of Calabi-Yau category is a Morita invariant notion, however we are not aware of a ``classical"  monoidal category 
interpretation of this strong notion of cyclic trace. 

With these definitions in mind, here is an abstract statement of our main results, in parallel to the
above highlighted statements for finite groups.

\begin{theorem}\label{intro Hecke thm}

1) (Theorems \ref{perfect dualizable},
    \ref{perfect Hecke}, \ref{CY Hecke}) The Hecke categories $\cH_G$ and $\cHtil_G$ are semi-rigid, and carry canonical pivotal and Calabi-Yau structures.


2) (Theorem \ref{thm char shvs}) The dg category
  $\chsh_G$ of unipotent character sheaves is canonically equivalent to the monoidal center and trace of both Hecke categories $\cH_G$ and $\cHtil_G$. 
 \end{theorem}

\begin{remark}
Perhaps the most interesting aspect of the identification
of $ \chsh_G$ with the monoidal trace of $\cH_G$ is the consequence that every dualizable
Hecke module $M$ has a character $\chi_M\in \chsh_G$. In other words, character
sheaves arise as characters of such Hecke modules. 
To
illustrate this point, consider a spherical subgroup $K\subset G$ (for
example, $K$ a parabolic subgroup, or the fixed points of an
involution).  Harish Chandra $(\fg,K)$-modules with trivial
infinitesimal character are equivalent to $K$-equivariant $\D$-modules
on $G/B$. The latter naturally provide a Hecke module where the Hecke
category acts via convolution on the right.  The character of this
Hecke module is the object of $\chsh_G$ given by pushing forward the
structure sheaf of ${K}\adjquot {K}$ along the natural projection to
${G}\adjquot {G}$, and then projecting it onto $\chsh_G$. 
For example, the
Springer sheaf $\cS_G$ discussed in Example~\ref{springer} is the
character of the left regular representation of the Hecke category.
\end{remark}

\begin{remark}
The identifications of the theorem endow $\chsh_G$ with rich algebraic
structures. As a monoidal center, $\chsh_G$ is
naturally an $\cE_2$-algebra. 
As
a monoidal trace, $\chsh_G$ comes equipped with an $S^1$-action.
In fact, the $S^1$-action is 
compatible with the loop rotation action on all $\D$-modules
on the loop space $ G\adjquot G=\cL(BG)$.
This is part
of the topological field theory structure we discuss in Section~\ref{TFT} below.
\end{remark}

\begin{remark}
It is interesting to compare the theorem with the independent results of
Bezrukavnikov, Finkelberg and Ostrik~\cite{BFO} on the level of abelian
categories. Namely, they show that the
Drinfeld center of the {abelian} category of Harish Chandra bimodules
is equivalent to the abelian category generated by character
sheaves. 
On the one hand, we do not check the compatibility of our constructions with $t$-structures;
on the other hand, the derived categories we consider are richer than the derived categories of their hearts.
 \end{remark}

\sssn{Langlands duality for character sheaves}

The Hecke categories $\cH_G$ and $\cHtil_G$ are connected by a
Langlands duality, namely the Koszul duality  of
Beilinson-Ginzburg-Soergel~\cite{BGS}, in the Langlands dual form
developed by Soergel~\cite{S} with monoidal structure established by
Bezrukavnikov-Yun~\cite{BY}.  It is most naturally stated in the
mixed setting, but since such structures are beyond the scope of this
paper, we will state an equivalence of underlying two-periodic
categories.  
Passing to two-periodic localizations
is a symmetric monoidal functor
$$
\xymatrix{
\{-\}^{per}:\St_\CC\ar[r] &  \St_{\CC[u,u\inv]} &  \cC\ar@{|->}[r] & \cC^{per}=\cC\otimes_{\CC} \CC[u,u\inv] 
& \deg u = 2
}
$$ 
 compatible with the formation of monoidal traces.
For $G$ and its Langlands dual group
$G^\vee$, the Koszul duality of~\cite{BGS, S, BY} localizes to an
equivalence of two-periodic Hecke categories
$$ \cHtil_{G}^{per}\simeq \cH_{G^\vee}^{per} 
$$ 
Thus Theorem~\ref{intro Hecke thm} implies  
a two-periodic equivalence between unipotent character sheaves on Langlands dual groups.

\begin{cor}
There is a canonical equivalence of two-periodic dg categories of
character sheaves on Langlands dual groups
$$ \chsh_{G}^{per} \simeq \chsh_{G^\vee}^{per}.
$$
\end{cor}

\ssn{Functional analysis of $\D$-module categories}\label{tools}
The proof of Theorem~\ref{intro Hecke thm} relies on 
the formalism of $\D$-module on stacks, as developed in~\cite{finiteness,dennisnick}, see Section~\ref{D modules}.
Our main interest is in results relating functors between $\D$-module categories
with integral transforms: given varieties $X,Y$ and a $\D$-module $\cK$ on the product
$X\times Y$, one defines a functor on derived categories of
$\D$-modules
 $$
 \xymatrix{
 \D(X)\ar[r] & \D(Y)
 &
 \cF\ar@{|->}[r] & \pi_{Y*}(\pi_X^!\cF\otimes \cK).
 }
 $$
by pulling back from $X$ to the product $X\times Y$, tensoring with the integral kernel $\cK$,
 and then pushing forward to $Y$
via the natural diagram
$$
\xymatrix{X &X \ti Y\ar[l]_-{\pi_X} \ar[r]^-{\pi_Y}& Y.
}
$$

\sssn{Integral transforms for quasicoherent sheaves}
In the context of quasicoherent sheaves, there is a very tight connection between
categorical operations and geometric operations on the underlying stacks.
In our paper \cite{BFN} with John Francis, we studied the
homotopical algebra of categories of quasicoherent sheaves, and in
particular proved a result identifying functors on quasicoherent
sheaves with integral transforms, generalizing results of Orlov~\cite{Orlov},
Bondal, Larsen and Lunts~\cite{BLL} and To\"en~\cite{Toen
  dg}. Namely, we
introduced the class of {perfect stacks} $X$, which as mentioned above are characterized (among quasi-compact stacks
with affine diagonal) by the property that $\qc(X)$ is compactly generated and rigid. This class
 includes all (quasi-compact and separated)
schemes as well as most common stacks in characteristic zero, and is
closed under natural operations such as fiber
products. 
For perfect stacks $X_1$,
$X_2$ over a perfect stack $Y$ we showed that the natural maps are equivalences
$$
\xymatrix{
\qc(X_1)\ot_{\qc(Y)} \qc(X_2) \ar[r]^-\sim & \qc(X_1\times_Y X_2)
\ar[r]^-\sim & \Fun^L_{\qc(Y)} (\qc(X_1) , \qc({X_2})).
}
$$
In other words, $\qc(Y)$-linear integral transforms  between $\qc(X_1)$ and $\qc(X_2)$ are identified
both with all $\qc(Y)$-linear functors  and with the relative tensor product of the 
respective categories. 

When $X = X_1 = X_2$, we also studied the algebra structure of convolution on such ``matrix" algebras of quasicoherent
sheaves. The dg category of integral transforms $\qc(X\times_Y
X)\simeq\Fun_{\qc(Y)}(\qc(X),\qc(X))$ has a natural multiplication,
given by convolution, or equivalently by
composition of functors, making it into a monoidal dg category.  In
the case $Y$ is a point, this is a categorified version of the algebra
of matrices with entries labelled by $X$, while in general it is a
version of $Y$-block diagonal matrices. Thus one expects the monoidal center and trace of such an algebra to
be a categorified version of
functions on $Y$.  Indeed, we showed that they are identified with the category $\qc(\cL Y)$ of sheaves on
the derived loop space of $Y$. (In fact $\qc(X\times_Y X)$ is Morita equivalent to $\qc(Y)$ under mild hypotheses, see~\cite{BFN appendix}.)


\begin{example} If we take $X=pt/B$ to be the classifying stack of the Borel and $Y=pt/G$, the center and trace of the quasicoherent Hecke category $\qc(B\bs G/B)$ are identified with all quasicoherent sheaves $\qc(G\adjquot G)$.
\end{example}

\sssn{Integral transforms for equivariant $\D$-modules}\label{Tannakian failure}
A modification of the techniques of \cite{BFN} provides an analogous description of integral transforms
in the $\D$-module setting when the base $Y$ is a scheme and the projection $X_1\to Y$ is a relative Deligne-Mumford stack,
or more generally, a {\em safe} morphism in the sense of \cite{finiteness}.

\begin{theorem}[Corollary~\ref{cor selfdual}] Let $X_1\to Y$ denote a Deligne-Mumford stack over a scheme, and $X_2\to Y$  an arbitrary stack. 
Then the natural maps are equivalences
$$
\xymatrix{
\D(X_1) \ot_{\D(Y)} \D(X_2) 
\ar[r]^-\sim &  
\D(X_1\ti_Y X_2) 
\ar[r]^-\sim & 
\Fun^L_{\D(Y)}(\D(X_1), \D(X_2))
}$$
\end{theorem}

The identification of tensors, integral transforms and functors
as above fails badly when the base is a stack, as in our motivating case $Y=BG$.
The fundamental obstruction is the inherently topological
nature of $\D$-modules.  Pushing forward along an affine map almost
always loses information (that is, it fails to be conservative): for
example, unlike quasicoherent sheaves, many nontrivial $\D$-modules
have no global flat sections at all.  Thus even for stacks with affine
diagonal, one can not reconstruct $\D$-modules on a fiber product from
algebraic operations on $\D$-modules on the factors (for stacks with
{finite} diagonal, for example for schemes, this problem disappears).
In short, the Tannakian theory of $\D$-modules is not rich enough to
capture the geometric theory.  As a result, tensor and functor
categories are not identified with all integral kernels: there
are adjunctions exhibiting the former as pale shadows of the latter.
For example, the category $\D(G)$ cannot be reconstructed from
$\D(pt)=\Mod_\CC$ as a module category for
$\D(BG)=C_*(G)\on{-mod}$ where $C_*(G)$ is the algebra of singular chains on $G$.

In Section~\ref{over stacks}, we prove the following partial result. We make the very strong assumption that 
the stack $Y$ is a classifying stack
of an affine group, thus ensuring that the technical conditions that $Y$ is smooth with affine diagonal. This setup  suffices for $\D(Y)$
to be semi-rigid under tensor product, a condition which typically fails  for schemes: compact objects in $\D(Y)$ for a scheme
 are coherent $\D$-modules, but only flat vector bundles are dualizable.
In this restricted setting, we then find that tensor products and linear functors with source a relative Deligne-Mumford stack, or more generally safe stack,
are identified with each other via the natural map factoring through a full inclusion into  integral transforms.

\begin{theorem}[Corollary~\ref{selfdual over stack}, Proposition~\ref{prop int trans adjoint}]
Let $X_1\to Y=BG$ be a relative Deligne-Mumford stack over the classifying stack of an affine group, and $X_2\to Y$ an arbitrary stack.
Then the natural map from tensors to functors is an equivalence, and factors through the integral transform construction
$$\xymatrix{&\D(X_1 \times_Y X_2)\ar[dr]^-{\Phi}&\\
\D(X_1)\ot_{\D(Y)} \D(X_2) \ar[rr]^-{\sim} \ar[ur]^-{\Psi}  &&\Fun^L_{\D(Y)}(\D(X_1), \D(X_2))}$$
Moreover, $\Phi$ has a fully faithful left adjoint.
\end{theorem}

\begin{example}
If $X_1=X_2=pt\to Y=BG$, so that $X\times_Y X=G$, we find that the tensor product and functor categories
are identified with the full subcategory $$\End_{\D(BG)}(\D(pt))\simeq \langle {\cO_G} \rangle \subset \D(G)$$ 
generated by the structure sheaf.
\end{example}

Using this more subtle understanding of the relation between integral transforms and functors, we calculate
the center and trace of  Hecke categories. Unlike the quasicoherent case, we do not see all sheaves on $\cL (BG)=G\adjquot G$ as the center and trace, but precisely the full subcategory of sheaves that can be accessed through the horocycle correspondence, or in other words, unipotent character sheaves.

\ssn{Topological field theory}\label{TFT}

 Our results naturally fit into and were motivated by the structure
 of topological field theory (TFT).  We will not give a formal introduction
 to TFT but rather remind the reader of its broad
 outline informally and illustrate it via the example of finite group
 gauge theory. We refer to \cite{jacob TFT} for a detailed overview of
 the $\oo$-categorical setting for TFT (including
 the $\oo$-version of $n$-categories) and the Cobordism Hypothesis. We only use the $(\infty,2)$-categorical language 
 for motivation and it does not recur elsewhere in the paper.

\subsubsection{2-dimensional TFT}

An extended framed 2-dimensional TFT is a symmetric monoidal functor
out of the symmetric monoidal $(\oo,2)$-category $2Bord^{fr}$ of
2-dimensional framed bordisms (with monoidal structure given by disjoint union):
\begin{enumerate}
\item[$\bullet$] objects: ($2$-dimensionally
  framed) $0$-manifolds, 
\item[$\bullet$] 1-morphisms: ($2$-dimensionally framed) $1$-dimensional bordisms
  between $0$-manifolds,
\item[$\bullet$] 2-morphisms: classifying spaces of (framed)
  $2$-dimensional bordisms between $1$-bordisms.
\end{enumerate}
In other words, all $n$-morphisms for $n>2$ are invertible, and
together form (the fundamental $\infty$-groupoid of) classifying
spaces of framed bordisms.
 
 One can similarly define the oriented 
bordism l $(\oo,2)$-category $2Bord$, in which one
extends from framed manifolds to all oriented manifolds and their
bordisms. 

In general, if the target is a plain discrete $2$-category, then one
obtains an ordinary functor out of the $2$-category with the same
objects and $1$-morphisms but $2$-morphisms given by 2-dimenisonal
bordisms up to differomorphism. This setting is studied in detail in \cite{schommer}. 

In \cite{jacob TFT}, Lurie outlines
the proof of a general $\oo$-categorical version of the Cobordism
Hypothesis formulated by Baez-Dolan \cite{baezdolan}. The
2-dimensional case which we use was inspired by work of Costello
\cite{Costello} and Kontsevich-Soibelman \cite{KS} on open-closed field theory.  The
Cobordism Hypothesis specifies increasingly strong finiteness conditions, known as $d$-dualizability,
 that allow an object  
of
a symmetric monoidal $(\oo, d)$-category to be ``integrated'' over
 framed manifolds of dimension at
most $d$, thus defining an extended $d$-dimensional TFT.
Moreover, the space of $d$-dualizable objects carries an action of
$O(d)$, and objects on which the action of $SO(d)$ has been trivialized give
rise to field theories defined on all oriented manifolds of
dimension at most $d$.

\subsubsection{Toy model: finite group gauge theory}

To place our results in the framework of TFT, it is illuminating to return to the toy model of a finite group $\Gamma$. 
There is an oriented 2-dimensional TFT  $Z_\Gamma$ called Dijkgraaf-Witten
theory~\cite{DW} which  efficiently
encodes all of the familiar structures in the complex
representation theory of $\Gamma$ discussed earlier.
(One could consult \cite{schommer} for details in all of the discussion to follow, in particular for the construction of 2-dimensional oriented TFTs
from symmetric Frobenius algebras.
See also \cite{Freed} for a detailed study of $Z_\Gamma$,
 and \cite{FHLT} for a discussion of $Z_\Gamma$ in the setting of the
Cobordism Hypothesis.)

The target of $Z_\Gamma$
is the Morita $2$-category of 
algebras $Alg_\CC$ with the following structure:
\begin{enumerate}
\item[$\bullet$] objects: associative algebras over $\CC$,
\item[$\bullet$] 1-morphisms: bimodules (flat over the source),
\item[$\bullet$] 2-morphisms: morphisms of bimodules.
\end{enumerate}
The symmetric monoidal structure of $Alg_\CC$ is given by tensor
product of algebras with unit $\CC$.
By assigning to an algebra $A$ its category $\Perf_A$ of perfect
modules (summands of finite colimits of free modules), we can identify the
Morita 2-category $Alg_\CC$ with a full subcategory of the 2-category
$AbCat_\CC$ of small $\CC$-linear abelian categories, exact functors
and natural transformations.

The field theory $Z_\Gamma$ assigns to each cobordism $M$ a
linearization of the space of $\Gamma$-gauge fields on $M$, or in other words, the
orbifold of principal $\Gamma$-bundles or Galois $\Gamma$-covers over
$M$.  In particular,  it assigns the following to closed $0$,
$1$ and $2$-manifolds:
\begin{enumerate}
\item[$\bullet$] To a point, $Z_\Gamma$ assigns the group algebra:
$$
Z_\Gamma(pt)=\CC [\Gamma]\in Alg_\CC.
$$ 
Alternatively, we can pass to the category of finite-dimensional complex modules:
$$
Z_\Gamma(pt) = \Rep^{\findim}_\CC(\Gamma) \in AbCat_\CC.
$$
This is the category of finite-dimensional algebraic vector bundles on the
orbifold of $\Gamma$-bundles on a point.

\item[$\bullet$] To a circle, $Z_\Gamma$ assigns 
  class functions:
$$
Z_\Gamma (S^1)=\CC [\Gamma]^\Gamma=\CC[\Gamma\adjquot {\Gamma}] \in \Vect^{\findim}_\CC =1\Hom_{AbCat_\CC}(\Vect^{\findim}_\CC, \Vect^{\findim}_\CC),
$$
This is the vector space of functions on the orbifold of $\Gamma$-bundles on the circle.

\item[$\bullet$] 
To a closed surface, $Z_\Gamma$ counts $\Gamma$-bundles:
$$
Z_\Gamma(\Sigma)=\# \{\Hom(\pi_1(\Sigma),\Gamma)/\Gamma\}\in \CC =  2\Hom_{AbCat_\CC}(\CC,\CC)
$$ 
where as usual a bundle $\cP$ is weighted by $1/\Aut(\cP)$. This is the volume of  
the orbifold of $\Gamma$-bundles on the surface.
\end{enumerate}

From the point of view of the Cobordism Hypothesis, we only have to
specify that we assign the group algebra $\CC[\Gamma]$ to a point to
determine the rest of the TFT structure. Any object of $Alg_\CC$ is 1-dualizable, with dual given by 
the opposite algebra. The requirement of 2-dualizability however is extremely restrictive: as proven in \cite{schommer}, 
2-dualizable objects of $ Alg_\CC$ are precisely separable algebras, i.e., algebras for which $A$ is projective as an $A$-bimodule.
Over $\CC$, separable algebras are precisely finite-dimensional semi-simple algebras.
Invariance under $SO(2)$ amounts to the data of a non-degenerate trace, or in other words, the structure of a symmetric Frobenius algebra. 


\subsubsection{TFTs from semi-rigid categories}

Our main results can be viewed as the construction and partial description
of an oriented 2-dimensional TFT valued in a categorified analogue of the Morita 2-category
$Alg_\CC$. Namely, we replace $\CC$-algebras by $\CC$-linear monoidal categories, or more precisely, 
algebra objects in  the symmetric monoidal
$\oo$-category $\St_\CC$. We denote by $Alg_{(1)}(\St_\CC)$ the Morita $(\oo,2)$-category with
objects algebras in  $\St_\CC$, 1-morphisms bimodule categories and
2-morphisms the spaces of intertwiners between bimodule categories.
(See \cite{jacob TFT} for a more precise description; we reiterate that we only use the
$(\oo,2)$-language for informal motivation). 

As with $Alg_\CC$, any object of $Alg_{(1)}(\St_\CC)$ is 1-dualizable,
with dual given by the monoidal opposite.
The much stronger condition of 2-dualizability of $A\in Alg_{(1)}(\St_\CC)$
breaks down into two parts. First, the underlying category of $A$ must be dualizable
as a plain category. In practice, this is  an easily satisfied condition, guaranteed for example by the compact generation of 
$A$, in which case the dual is a modification of
the opposite category of $A$. 
Second, one needs $A$ to be dualizable as an $A$-bimodule,
a natural analogue of separability. Dualizability as a bimodule 
is closely related to the abundance of monoidal duals of objects of $A$:

\begin{theorem}[Theorem \ref{perfect dualizable}]
Any semi-rigid category $A\in Alg_{(1)}(\St_\CC)$ is 2-dualizable. 
\end{theorem}

It follows that any semi-rigid monoidal category $A\in Alg_{(1)}(\St_\CC)$, in particular the Hecke categories $\cH_G$ and $\cHtil_G$, defines
a framed 2-dimensional TFT $\cZ_A$ valued in $Alg_{(1)}(\St_\CC)$ which makes the following assignments:

\begin{enumerate}
\item[$\bullet$] To a point, $\cZ_A$ assigns $A\in Alg_{(1)}(\St_\CC)$.
\item[$\bullet$] To a circle with cylinder framing, $\cZ_A$ assigns the trace $\Tr(A)\in \St_\CC$ with its cyclic $S^1$-action. 
\item[$\bullet$] To the circle with annulus framing, $\cZ_A$ assigns the center $\Z(A) \in \St_\CC$ with its  $\cE_2$-structure.
\item[$\bullet$] To a framed surface, $\cZ_A$ assigns an object of $\Vect_\CC$.
\end{enumerate}

\subsubsection{Oriented TFTs from Calabi-Yau monoidal categories}

To extend the  TFT $\cZ_A$ from framed to oriented manifolds,
the Cobordism Hypothesis requires $A\in Alg_{(1)}(\St_\CC)$ be a fixed point for the $SO(2)$-action on 2-dualizable objects.

Part of this action can be described very explicitly: following 
\cite[Proposition 4.2.3]{jacob TFT},
a 2-dualizable object $A\in Alg_{(1)}(\St_\CC)$ admits a Serre automorphism $S_A$ characterized by
$$ev^R_A = (S_A\ot \Id_{A^{op}})\circ coev_A$$
It can be interpreted as the monodromy of $A$ around the $SO(2)$-action, or in other words, the first obstruction
to $SO(2)$-invariance of $A$. Thus a trivialization of $S_A$ as an automorphism makes $A$ a fixed point of the action of the free group (James construction) on the circle $\Omega S^2=\Omega\Sigma SO(2)$. This suffices to identify the values of $\cZ_A$ on different framed circles, for example the monoidal center and trace of $A$.

Following \cite[Section 4.2]{jacob TFT}, the stronger condition of $SO(2)$-invariance of $A$ can be encoded by the notion of 
Calabi-Yau monoidal category (as recalled in Definition~\ref{CY def} above).

\begin{cor}[Theorem \ref{perfect dualizable}]\label{TFT corollary}
A Calabi-Yau (respectively, pivotal) structure on a semi-rigid category $A\in  Alg_{(1)}(\St_\CC)$ defines an oriented
 (respectively, a weakly oriented)
2-dimensional TFT
so that $\cZ_A(S^1)$ is identified with both the
monoidal center and trace of $A$.  
\end{cor}

\begin{example}[Fusion categories and 3-dimensional Dijkgraaf-Witten theory] 
For a finite group  $\Gamma$, the categories $\Vect^\findim_\CC(\Gamma)$ of finite-dimensional vector bundles on $\Gamma$ and $\Rep_\CC^{\findim}(\Gamma)$ of finite-dimensional complex representations of $\Gamma$ are key examples of fusion categories, monoidal categories satisfying strong finiteness conditions. In fact, it is proved in \cite{DSS} that fusion categories are 3-dualizable objects of a natural 3-category of monoidal categories, bimodule categories, intertwines and natural transformations. (They also show that pivotal structures give $\Omega S^2$-fixed structures on these theories.) The monoidal categories 
$\Vect^\findim_\CC(\Gamma)$ and $\Rep_\CC^{\findim}(\Gamma)$
are Morita equivalent and so define the same 3-dimensional TFT, an untwisted version of 3-dimensional Dijkgraaf-Witten theory \cite{DW} and a natural categorification of the finite group gauge theory discussed earlier. Since 
$\Rep_\CC^{\findim}(\Gamma)$  is symmetric monoidal, it is easy to evaluate this TFT on any manifold.
For example, to a surface $\Sigma$, one attaches the vector space of functions on the orbifold of $\Gamma$-principal bundles over $\Sigma$, while to a 3-manifold $M$, one  attaches the volume of the orbifold of $\Gamma$-principal bundles over~$M$.
\end{example}

\begin{example}[Quasicoherent 3-dimensional theories]
For a perfect stack $X$,  the symmetric monoidal category $\qc(X)$ of quasicoherent sheaves is rigid and pivotal. 
For example, for $G$ a reductive group,  its classifying stack $BG$ is perfect,
and the category $\qc(BG)$ comprises  algebraic
representations of~$G$. 

One can calculate the resulting categorified 2-dimensional TFT  explicitly by the theory of topological chiral
homology \cite[Section 5.3]{HA}, or from the direct constructions of
\cite{BFN}. In particular,  it makes the following assignments,
stated in general, then spelled out when  $X=BG$:
\begin{enumerate}
\item[$\bullet$] $Z_{X}^{qc}(pt)=\qc(X) =\Rep(G)$, algebraic representations.
\item[$\bullet$] $Z_{X}^{qc}(S^1)=\qc(\cL X)=\qc(G\adjquot G)$, adjoint-equivariant
  quasicoherent sheaves. 

\item[$\bullet$] $Z_X^{qc}(\Sigma)=\cO(X^\Sigma) = R\Gamma(Char_G(\Sigma),\cO)$,
  functions on the character variety.
  
%
\end{enumerate}

As explained in \cite{BFN appendix}, the rigid category
$\qc(BG)$ is in fact Morita equivalent to
the quasicoherent group algebra $\qc(G)$, and any intermediate Hecke category
$\qc(H\bs G/H)$, for $H\subset G$, so that all define equivalent TFTs.

Furthermore, $\qc(X)$ is naturally a Calabi-Yau monoidal category, and hence $Z_{X}^{qc}$ extends to an oriented TFT.
However, unlike the  case of fusion categories,  $\qc(X)$ is not typically 3-dualizable,
and hence
$Z_{X}^{qc}$ does not typically extend to 3-manifolds.
Nonetheless, it is natural to think of our constructions as intrinsically 3-dimensional as they correspond to supersymmetric 3-dimensional quantum gauge theories in physics.
\end{example}

\subsubsection{Character Theory}
Finally, let us turn to TFTs built out of $\D$-modules
on the complex reductive group $G$. In analogy with the above results for
quasicoherent sheaves, one  might expect to construct a single TFT starting from 
a Morita equivalent class of algebras including the ``commutative algebra"
$\D(BG)$, the ``smooth group algebra'' $\D(G)$, or any intermediate Hecke algebra
$\D(H\bs G/H)$, for $H\subset G$. Furthermore, one might expect  this TFT to attach to the circle the category 
$\D(G/G)$ of
 adjoint-equivariant $\D$-modules on $G$, and to a surface
$\Sigma$ the de Rham cohomology
$\Gamma(Char_G(\Sigma),\Omega^\bullet)$ of the character
variety.

This expectation turns out to be wrong on several counts, starting
with the failure of Morita equivalence. The symmetric monoidal category $\D(BG)$ is
indeed 2-dualizable, however it is too small to see the rich geometry
 of $G\adjquot G$,  as discussed in Section~\ref{Tannakian failure} above. It only sees  the homotopy type of $G$, and in particular assigns to the circle the category of adjoint-equivariant unipotent local systems on $G$.  At the other
extreme, the monoidal category $\D(G)\in Alg(\St_\CC)$ is not 2-dualizable, so does not define a
2-dimensional TFT, though its center is indeed 
$\D(G/G)$.

There is a sweet spot in between these two extremes given by Hecke categories.
The main results of this paper allow for the following definition and calculate the resulting TFTs on a circle.
%

\begin{definition} The  {\em equivariant character theory} $\chi_G$ 
of a reductive group $G$ is the oriented 2d TFT valued in $Alg_{(1)}(\St_\CC)$
defined by the Hecke category $\cH_G$. Likewise, the {\em monodromic character theory} $\widetilde{\chi}_G$ is the oriented 2d TFT
valued in $Alg_{(1)}(\St_\CC)$ defined by $\cHtil_G$.
\end{definition}

It follows from the Koszul duality of~\cite{S, S, BY} that there is a Langlands duality for the two-periodic versions of the above TFTs.

\begin{cor} The invariants assigned by $\widetilde{\chi}_G^{per}$ and $\chi_{\Gv}^{per}$ to
any oriented manifold are identified.
\end{cor}

In particular, unipotent character sheaves with their TFT structures are identified by Langlands duality.
In another direction, boundary conditions for the character theory are given by module categories for the corresponding Hecke categories, for 
example, by categories $\D(K\backslash G/B)$ of Harish Chandra $(\fg,K)$-modules. The Langlands duality of boundary conditions in particular encodes Soergel's conjecture~\cite{S} on a categorical Langlands classification for real groups.
In another direction, it would be highly interesting to describe the vector
spaces $\chi_{G}(\Sigma)$ that the character theory attaches
to a surface $\Sigma$. It is natural to conjecture a relation with the
de Rham cohomology of character varieties, and hence with the
fascinating conjectures of Hausel, Letellier and Rodriguez-Villegas
\cite{HRV,HLRV,H}. Such a relation would  imply a Langlands duality
 for the cohomology of character varieties, one of our original motivations and 
 a subject of ongoing work with Sam Gunningham.

\subsection{Further discussion}\label{families}
We conclude the introduction 
with a brief discussion of twisted versions of our results and relations with supersymmetric gauge theory.

\subsubsection{Twisted $\D$-modules and families version}

The results of this paper have natural analogues in which $\D$-modules on flag varieties are replaced by
twisted $\D$-modules. The requisite arguments are formally identical, but we will only briefly summarize the picture here. We
expect a more detailed account to appear in forthcoming work with Sam Gunningham.

The monodromic Hecke category $\cHtil_G$ fits into a family of twisted monodromic Hecke categories $\cHtil_{G,\lambda}$, indexed by characters  $\lambda\in \fh^{\vee}$. 
Likewise, the dg category $\chsh_G$ of unipotent character sheaves fits into a family of dg categories
 $\chsh_{G, [\lambda]}$ of character sheaves with central character $[\lambda]\in H^\vee/W$.
 %

\begin{expectation}
Fix any $\lambda\in \fh^\vee$, with image $[\lambda]\in H^\vee/W$.

1) The twisted monodromic Hecke category $\cHtil_{G, \lambda}$ is a semi-rigid Calabi-Yau monoidal category.

2) The monoidal center and trace of
  $\cHtil_{G, \lambda}$ are both  identified with
  character sheaves $\chsh_{G, [\lambda]}$ with central character $[\lambda]\in H^\vee/W$.

\end{expectation}
%
%
%

The  Koszul duality of Hecke categories extends in an interesting
way to other characters $\lambda\in\fh^\vee$. On the monodromic side, one has the twisted 
monodromic Hecke category $\cHtil_{G, \lambda}$.
On the equivariant side, one  introduces  the centralizer  $G^\vee_\lambda \subset G^\vee$ of a semi-simple
representative of $\lambda$, a Borel subgroup $B^\vee_\lambda\subset
G^\vee_\lambda$,  and the equivariant Hecke
category $\cH_{G^\vee_\lambda} = \D(B^\vee_\lambda\bs
G^\vee_\lambda/B^\vee_\lambda)$.
The results of \cite{BGS, S, BY} give an equivalence of monoidal categories 
$$\cHtil_{G, \lambda}^{per}\simeq \cH_{G^\vee_\lambda}^{per}
$$

Now if the above expectation is realized, then  one can introduce a twisted 2-dimensional character theory $\widetilde{\chi}_{G,\lambda}$.
It will satisfy a Langlands duality
$$
\widetilde{\chi}_{G,\lambda}^{per} \simeq \chi^{per}_{G^\vee_\lambda}
$$ 
and in particular, its value on the circle will provide an equivalence
$$ \chsh_{G, [\lambda]}^{per} \simeq \chsh_{G^\vee_\lambda}^{per}.
$$
relating character sheaves on $G$
with central character $[\lambda]$ and unipotent character sheaves on 
the centralizer $G^\vee_\lambda$ in the dual group.

\begin{remark}
At first encounter, the idea (derived from Koszul duality) to consider the Hecke categories
$\cH_{G^\vee_\lambda}$ as a family might seem bizarre.
However, in \cite{reps} we establish (using the
techniques developed in \cite{conns}) that the family of 
Hecke categories $\cH_{G^\vee_\lambda}$ arises naturally by
$S^1$-equivariant localization from an evident family of affine Hecke
categories. 
One can then see its Koszul duality with $\cHtil_{G,
  \lambda}$ as a shadow of the Langlands duality for affine Hecke categories
 due to Bezrukavnikov~\cite{Roma ICM}.
\end{remark}

\subsubsection{Relation with gauge theory and geometric
 Langlands}\label{physics} 
The character theory has an expected physical interpretation due to Witten~\cite[Section
  3.3]{Witten Atiyah}, roughly coming from a topological twist
of a certain maximally supersymmetric three-dimensional gauge theory. This relation is a dimensional reduction of the 
groundbreaking interpretation of the geometric Langlands correspondence as Montonen-Olive S-duality for 
maximally supersymmetric four-dimensional gauge theory
due to Kapustin and Witten~\cite{KW}. In particular, the Langlands duality for character theories should arise as a dimensionally reduced version of the geometric Langlands correspondence.
Namely, Witten considers 4-dimensional ${\mc N=4}$
super-Yang-Mills in the ``geometric Langlands'' topological twist as studied in \cite{KW} (for any value of the twist and coupling parameter)
reduced on 4-manifolds of the form $M^3\times S^1$. The supercharge in this theory has a perturbation 
which squares to rotation along the $S^1$-factor
(this is closely related to the ``$\Omega$-background'' in the form
studied in \cite{NW}, or to the appearance of $\D$-modules from cyclic homology discovered in \cite{cyclic}). 
This supercharge still defines a differential when
restricted to $S^1$-invariant modes, and can be used to define a
$\ZZ/2$-graded 3-dimensional TFT (independent of parameters), which is the physical construction of the character theory.
The parameter $\lambda\in \fh^\vee$ has the physical interpretation as a
vacuum expectation value (vev) of a scalar field, so that
$\fh^\vee/W$ appears as the moduli of vacua of the theory (much as the
Hitchin base appears as the moduli of vacua or Seiberg-Witten base for
supersymmetric gauge theories in four dimensions), with the unipotent version studied here arising at the conformal point.



\subsection{Acknowledgements} 
This work relies heavily on the ideas and results of Jacob Lurie, whom
we wish to thank for many patient explanations. We would like to thank
Edward Witten for invaluable extended discussions, in particular on
the topic of supersymmetric gauge theories and dimensional reduction.
Our work on this project benefited greatly from our collaboration
\cite{BFN} with John Francis and many independent discussions
with him, as well as from our ongoing collaboration with Sam
Gunningham, who provided many invaluable corrections and comments. We
would also like to thank Orit Davidovich, Toly Preygel, Nick Rozenblyum and Chris Schommer-Pries for
helpful discussions, Tom Nevins for a close reading and many
corrections of the first draft, Roman Bezrukavnikov for sharing his
ideas and works in progress, and David Kazhdan for his generous
interest and comments. Finally the authors would like to thank an anonymous referee for a careful reading and many
insightful comments. DBZ is partially
supported by NSF grants DMS-0449830 (CAREER) and DMS-1103525,  DN is partially supported
by NSF grants DMS-0600909, DMS-1319287 and a Sloan Research Fellowship. This work was supported in part by the 
National Science Foundation under Grant No. PHYS-1066293 and the hospitality of the Aspen Center for Physics, and by the IAS
supported by NSF grant DMS-0635607.




\sn{Preliminaries on $\oo$-categories}\label{prelim section}
In this section, we summarize relevant technical foundations for 
homotopical algebra with $\oo$-categories.
It is beyond the scope of this
paper to provide a comprehensive discussion of
these topics.  We will instead focus on highlighting the particular
concepts, results and their references that play a role in what
follows. Our primary source is \cite{HA}; for a less condensed overview, one could also
consult~\cite{BFN} and~\cite{DGcat}.

Throughout this paper, we will work over the complex numbers $\CC$ (or any fixed
algebraically closed field of characteristic zero).
We will study (pre-triangulated) dg categories by means of the underlying $\CC$-linear stable $\infty$-category. 
See~\cite{DGcat} for an $\infty$-categorical study of dg categories and \cite{cohn} for the precise identification of the homotopy theory of pre-triangulated
dg categories over a field $k$ with $k$-linear stable $\infty$-categories.
 

\subsection{$\infty$-category basics} \label{infinity}


Roughly speaking, an $\infty$-category (or synonymously $(\oo,
1)$-category) encodes the notion of a category whose morphisms form
topological spaces and whose compositions and associativity properties
are defined up to coherent homotopies.  The theory of
$\infty$-categories has many alternative formulations (see
\cite{Bergner} for a comparison between the different versions). We
will follow the conventions of \cite{topos}, which is based on Joyal's
quasi-categories \cite{Joyal}. Namely, an $\infty$-category is a
simplicial set satisfying a weak version of the Kan condition
guaranteeing the fillability of certain horns. The underlying vertices
play the role of the set of objects while the fillable horns
correspond to sequences of composable morphisms. The book \cite{topos}
presents a detailed study of $\infty$-categories, developing analogues
of many of the common notions of category theory (an overview of the
$\infty$-categorical language, including limits and colimits, appears
in \cite[Chapter 1.2]{topos}).  An important distinction between
$\infty$-categories and the more traditional settings of model
categories or homotopy categories is that coherent homotopies are
naturally built into the definitions. Thus for example, all functors
are derived and the natural notions of limits and colimits correspond
to {homotopy} limits and colimits. We will make essential use of the
theory of adjoint functors, monads and Lurie's
extension~\cite[6.2.2]{HA} of the Barr-Beck theorem to the setting of
$\oo$-categories. 

Most of the $\oo$-categories that we will encounter are
{presentable}~\cite[5.5]{topos} in the sense that they are closed
under all small colimits (as well as limits, by \cite[Proposition
  5.5.2.4]{topos}), and generated under suitable colimits by a small
category. Examples include $\oo$-categories of spaces and of modules
over a ring.  Presentable $\infty$-categories form an
$\infty$-category ${\mc Pr}$ whose morphisms are left adjoints, or
equivalently by the adjoint functor theorem \cite[5.2]{topos},
functors that preserve all colimits \cite[5.5.3]{topos}.  

An $\oo$-category is
{stable}~\cite[1.1,1.2]{HA} if it has a zero object,
is closed under finite limits and colimits, and pushouts and
pullbacks coincide.  Stable $\oo$-categories are an analogue of the
additive setting of homological algebra: the homotopy category of a
stable $\infty$-category has the canonical structure of a triangulated
category~\cite[1.1.2]{HA}.  We will denote by $\St\subset \mc Pr$ the
full $\oo$-subcategory of stable presentable $\oo$-categories as
studied in \cite[1.4.5]{HA}. 


\sssn{Ind-categories
and compact generators}\label{sect ind-categories} 

Recall that an object $c$ of an $\oo$-category $\cC$ is compact if the
functor $\Hom_\cC(c, -)$ preserves (small) colimits (see
\cite[5.3.4]{topos} for a detailed discussion).

Given a functor $f:\cC\to \cD$ between $\oo$-categories, we will say
\begin{enumerate}
\item[$\bullet$] $f$ is {\em continuous} if it preserves all (small)
  colimits, and
\item[$\bullet$] $f$ is {\em quasi-proper} if
it takes compact objects to compact objects.
\end{enumerate}

We briefly recall the properties of ind-categories
from \cite[Section 5.3.5]{topos}. Given a small $\oo$-category $\cC$ which admits finite colimits, 
we may freely adjoin to $\cC$ all small filtered colimits. The result is a new $\oo$-category $\Ind \cC$,
called the ind-category of $\cC$,
which is presentable (and so in particular, admits all small colimits). 
By \cite[1.1.3]{HA}, if
$\cC$ is stable then so is  $\Ind \cC$.
In general, $\Ind\cC$ can be identified with the category of those presheaves on $\cC$ 
taking finite colimits to finite limits (and so
in particular,
it comes with an embedding $\cC\subset \Ind \cC$).

Let us denote by $\Idem$ the $\oo$-category of small stable $\oo$-categories that are idempotent-complete
 \cite[4.4.5]{topos} (see \cite[4.1]{BFN} for a discussion). 
For $\cC\in\Idem$, we can 
 recover $\cC$  from its ind-category  $\Ind \cC$ as the full $\oo$-category $\cC\simeq (\Ind \cC)_{cpt}$ of compact objects of
 $\Ind \cC$ (see \cite[Lemma 5.4.2.4]{topos} or 
 \cite[5.3.4.17]{topos}, and also Neeman \cite{neemanTTY}). 

This hints at the intimate relation between ind-categories and the classical notion of compact generation
(see \cite[1.4.5]{HA} and \cite[5.5.7]{topos} for more details).
A stable $\oo$-category $\cC$ is said to be compactly generated
if it admits filtered colimits and there is a small $\infty$-category $\cC^\circ$ of compact objects $c_i\in
\cC$ whose right orthogonal vanishes: if $m\in\cC$ satisfies
$\on{\Hom}_\cC(c_i,m)\simeq 0$, for all $i$, then $m\simeq 0$.


Note that an ind-category $\Ind \cC$ is automatically compactly
generated (with compact objects the objects of $\cC$).  In fact, any
compactly generated stable presentable $\oo$-category $\cC$ can be
identified with the ind-category $\Ind \cC_{cpt}$ of its subcategory
$\cC_{cpt}\subset \cC$ of compact objects. This is a version of a
theorem of Schwede and Shipley \cite{schwedeshipley} characterizing
module categories for $A_\infty$-algebras (see \cite[7.1.2]{HA} for
the $\infty$-categorical version, and \cite{Keller} for the
differential graded version).  More generally, the functor
$\Ind:\Idem\to \St$ identifies $\Idem$ with the subcategory of $\St$
consisting of compactly generated $\oo$-categories, and with morphisms
quasi-proper continuous functors; a quasi-inverse is given by $\cC\mapsto
\cC_{cpt}$.


\subsubsection{Opposite categories and restricted opposites}\label{restricted op}
We denote by $\cC^\catop$ the opposite category of an $\oo$-category $\cC$
(reserving the superscript ${op}$ for the opposite monoidal structure on algebras;
see Section~\ref{monoidal infinity} below). 
We will typically consider opposite categories for small categories only. 
For $\cC=\Ind \cC_{cpt}$ a compactly generated $\oo$-category, we will work with the modified notion of the restricted opposite category
defined by 
$$\cC'=\Ind(\cC_{cpt}^\catop).
$$ The terminology is motivated by that of the restricted dual of vector spaces
equipped with an extra structure such as a grading. 

\begin{example}
For the $\oo$-category $\Vect_\CC$ of $\CC$-vector spaces, $\Vect_{\CC, cpt}$ consists
of perfect complexes of $\CC$-vector spaces. Duality gives an identification 
$\Vect_{\CC, cpt}\simeq \Vect_{\CC, cpt}^\catop$,
which extends by continuity to an equivalence $\Vect_{\CC}\simeq \Vect_{\CC}'$. 
Note that the plain opposite category $\Vect_{\CC}^\catop$ is the $\oo$-category 
of pro-finite dimensional vector spaces, which is quite different from $\Vect_\CC'$.
\end{example}

\subsubsection{Adjoint and quasi-proper functors}
We record here a couple of other useful statements.

\begin{lemma}\label{quasi-proper functors} Let $F:{\cC}\longleftrightarrow{\cD}: G$ denote an adjoint
pair of  functors of $\oo$-categories. If $G$ is continuous, then $F$ is quasi-proper.
Conversely, if ${\cC}$ is compactly generated and $F$ is quasi-proper, then $G$ is continuous.
\end{lemma}

\begin{proof}
See \cite[Proposition 1.2.4]{BunG compact}.
\end{proof}

We will only need the following in the stable setting though it holds more generally.

\begin{lemma}\label{Ind adjoint} Let $f:\cC\longleftrightarrow \cD: g$ denote an adjoint
pair of   exact functors between small stable $\oo$-categories. 
Then the induced functors $\Ind f:\Ind \cC \longleftrightarrow \Ind \cD: \Ind g$ are also
adjoint.
\end{lemma}

\begin{proof}
The assertion follows immediately from the universal property \cite[Proposition 5.3.5.10]{topos} of the $\Ind$ construction:
for $\cC$ a small stable $\oo$-category and $\cD\in \St$ a presentable stable $\oo$-category, there is an equivalence of $\oo$-categories
$$
\xymatrix{
\Ind: \Fun^{ex}(\cC,\cD)\ar[r]^-\sim & \Fun^L(\Ind \cC,\cD)
}$$
between the categories of continuous functors (equivalently, left adjoints)
 from $\Ind\cC$ to $\cD$ and functors
 from $\cC$ to $\cD$ that preserve finite colimits.
\end{proof}

\subsubsection{Reversing diagrams}\label{reversing diagrams}
The adjoint functor theorem has a very useful corollary that enables us to calculate suitable colimits of $\oo$-categories
by turning them into (much simpler) limits.
Recall that given a small colimit-preserving (respectively, small limit-preserving  accessible) functor 
$F:\cC\to \cD$ between stable presentable $\oo$-categories,
the adjoint functor theorem provides a right (respectively, left) adjoint. 
In other words, the opposite of the $\oo$-category $\St$ of stable presentable $\oo$-categories
with morphisms colimit-preserving functors is the $\oo$-category $\St^R$ of 
 stable presentable $\oo$-categories
with morphisms accessible limit-preserving functors (see (\cite[Corollary 5.5.3.4]{topos} in the unstable setting). 
Moreover, small limits
of stable presentable $\oo$-categories are independent of context: the
forgetful functors 
$$
\xymatrix{
\St\ar[r] & \Cat_\oo &  \St^R \ar[l] 
}
$$
 to the $\oo$-category $\Cat_\oo$ of all
$\oo$-categories preserve limits. Thus the limit of a
diagram in $\St$ consisting of limit-preserving functors can
be calculated as the limit of the same diagram regarded in $\St^R$. In turn, the limit of a
diagram in $\St^R$ can be calculated
 as the
colimit of the corresponding diagram in $\St$ formed by the left adjoints.

\subsection{Monoidal $\infty$-categories and duals}\label{monoidal infinity}
A monoidal $\infty$-category as defined in \cite[4.1.1]{HA} is an
$\oo$-category equipped with a homotopy coherent associative unital
product. Its homotopy category is an ordinary monoidal category.  An
algebra object $A$ in a monoidal $\oo$-category as defined
in~\cite[4.1.2]{HA} is an object equipped with a homotopy coherent
multiplication.  Left and right module objects over an algebra object
are defined similarly \cite[4.2.2]{HA}, with right modules identified
with left modules over the opposite algebra object $A^{op}$  in the opposite 
monoidal $\infty$-category. There
is a relative tensor product $\cdot \ot_A \cdot$ of left and right
modules given by the two-sided bar construction \cite[4.3.4]{HA}.
Monoidal $\infty$-categories, algebra objects in a monoidal
$\oo$-category, and module objects over an algebra object themselves
form $\infty$-categories, some of whose properties (in particular,
behavior of limits and colimits) are worked out in \cite[4.2]{HA}.


The definition of a symmetric monoidal $\infty$-category is given in
\cite[2.0.0.7]{HA} modeled on the Segal machine for infinite loop
spaces.  Its homotopy category is an ordinary symmetric monoidal
category.  There is the notion of commutative algebra object $A$ in a
symmetric monoidal $\oo$-category  such that $A$-modules form a
symmetric monoidal $\oo$-category with respect to relative tensor
product~ \cite[4.4.2]{HA}.  Given two associative algebra objects $A, B$, their monoidal product $A\ot B$ carries a
natural associative algebra structure.  Furthermore, any associative
algebra object $A$ is a left (as well as a right) module object over
the associative algebra $A\ot A^{op}$ via left and right
multiplication.

The $\infty$-category ${\mc Pr}$ of presentable $\infty$-categories
and continuous functors carries a symmetric monoidal structure
introduced in \cite[6.3.1]{HA}. The tensor product $\cC\ot\cD$ of
presentable $\cC,\cD$ is a recipient of a universal functor from the
Cartesian product $\cC\times \cD$ which is ``bilinear" (commutes with
colimits in each variable separately).  In fact, the tensor product
lifts to a symmetric monoidal structure in which the unit object is
the $\infty$-category of spaces~\cite[6.3.1]{HA}.  Furthermore, the
symmetric monoidal structure is closed in the sense that ${\mc Pr}$
admits an internal hom functor compatible with the tensor structure~
\cite[Remark 5.5.3.9]{topos}, \cite[Remark 6.3.1.17]{HA}.  The
internal hom assigns to presentable $\infty$-categories $\cC$ and
$\cD$ the $\infty$-category of colimit-preserving functors
$\Fun(\cC,\cD)$ which is presentable by \cite[Proposition
  5.5.3.8]{topos}.  
  
  The symmetric monoidal structure on the
$\infty$-category $\mc Pr$ of presentable $\infty$-categories
restricts to one on the full $\infty$-subcategory $\St \subset \mc Pr$ of stable
presentable $\infty$-categories~\cite[6.3.2]{HA}, \cite[6.3.2]{HA}.
The unit is the stable category, modules over the sphere spectrum.
By \cite[Proposition 4.4]{BFN}, the functor $\Ind:\Idem\to \St$ is in
a natural way a {symmetric monoidal} functor
$$\Ind(\cC_1\ot\cC_2)\simeq \Ind(\cC_1)\ot\Ind(\cC_2)$$ 
for the
natural symmetric monoidal structures on both sides.

A basic example of a symmetric monoidal stable $\oo$-category is $\Vect_\CC$, which underlies the dg category of chain complexes of $\CC$-vector spaces.
The rest of this paper takes place in the setting of the symmetric monoidal $\oo$-category 
$$\St_\CC=\Mod_{\Vect_\CC}(\St)$$
of $\CC$-linear stable presentable $\oo$-categories, defined as modules for $\Vect_\CC$ in stable categories.
The category $\St_\CC$ provides a model for the theory of differential graded categories over $\CC$ (see \cite{cohn} for the identification
of the corresponding homotopy theories). As the category of modules over a commutative algebra object (namely $\Vect_\CC$), $\St_\CC$ inherits a symmetric monoidal structure itself with unit $\Vect_\CC$.
All of the foundations and formal results for stable $\oo$-categories naturally
translate to the $\CC$-linear setting by base change to the commutative algebra $\Vect_\CC$.

\sssn{Dualizability}\label{duals}

Let $A$ denote a monoidal $\oo$-category with monoidal product $\star$ and monoidal unit $1_A$. A key property for objects
of $A$ is dualizability (see \cite[Section 4.2.5]{HA}).

\begin{definition} An object $a\in A$ of a monoidal $\oo$-category
 is {\em left dualizable} (respectively, {\em right dualizable}) if there
  is an object ${}^\vee a$ (respectively, $a^\vee$) together with morphisms
$$\xymatrix{
coev: 1_A\ar[r] &  {}^\vee a \star a
& 
ev: a \star {}^\vee a
  \ar[r] &  1_A
  }
  $$ so that the obvious compositions
$$
\xymatrix{
a\ar[r] & a\star {}^\vee a \star a \ar[r] & a &
      {}^\vee a\ar[r] & {}^\vee a \star a \star {}^\vee a \ar[r] &
      {}^\vee a}
      $$ are equivalent to the corresponding identity maps
      (respectively,  
      $$\xymatrix{
      coev:1_A\ar[r] &  a\star a^\vee  &
      ev: a^\vee \star a\ar[r] & 1_A
      }$$ satisfying analogous composition
      identities).
\end{definition}

Note that the left dual of $a^\vee$ is again $a$,
and the right dual of ${}^\vee a$ is again $a$, i.e., left and right dualities are inverses of each other.
For a {symmetric} monoidal $\oo$-category $A$, we will take advantage of the
natural identification between the notions of left and right duals and
speak simply of the dual $a'$.

One can reinterpret the notion of dualizable object in terms of adjoint functors: 
 $a\in A$ 
 is {left dualizable} (respectively, {right dualizable}) if the functor $ \star a$ admits the left adjoint
 $\star {}^\vee a  $, or equivalently
  the functor $ a \star$ admits the right adjoint
 ${}^\vee a \star  $
(respectively, if the functor $ \star a $ admits the right adjoint
 $\star a^\vee  $, or equivalently
  the functor $a\star $ admits the left adjoint
 $a^\vee\star  $).
 Let us record this in the following form for easy reference.

\begin{lemma}\label{internal}  
For a right (respectively, left) dualizable object $a$ in a monoidal
$\oo$-category $A$ we have functorial identifications
$$
\xymatrix{
\Hom_A(a^\vee\star b, c)\simeq \Hom_A( b,a\star c)
&
\Hom_A( b\star a, c)\simeq \Hom_A( b, c \star a^\vee)
}
$$ 
which are monoidal in $a$.

Likewise, for an $A$-module category $M$ and objects $m,n\in M$, we have a functorial identification
$$
\xymatrix{
\Hom_M(a^\vee\star  m, n)\simeq \Hom_M( m, a \star  n)
}
$$
where we write $\star$ for the $A$-action on $M$.

\end{lemma}

\sssn{Dualizability of stable $\infty$-categories}
We work with stable presentable $\CC$-linear categories $\St_\CC$ (or dg categories), though everything works in the general setting $\St$ over the sphere.
As discussed above, the $\oo$-category $\St_\CC$ is naturally symmetric monoidal with unit the $\oo$-category $\Vect_\CC$ of chain complexes.

\begin{prop}\label{jacob duality}
Suppose $\cC\in \St$ is compactly generated so that $\cC=\Ind
\cC_{cpt}$. Then $\cC$ is dualizable in $\St$ with dual the restricted
opposite $\cC'=\Ind(\cC_{cpt}^\catop)$.
\end{prop}

\begin{proof}
See \cite[Proposition 2.3.1]{DGcat}.
\end{proof}

\subsubsection{Dualizability of modules}
We will also require a slightly generalized notion of dualizablity,
which applies in the (non-monoidal) setting of modules over an
associative algebra. Let $A\in \Alg(\cS)$ denote an associative algebra in a
symmetric monoidal $\oo$-category $\cS$ with unit $1_\cS$, and let $\cC=A\otimes A^{op}\on{-mod}$ denote the $\oo$-category of $A$-bimodules in $\cS$.

\begin{definition} \label{duality for modules}
An $A$-module $ M$ is said to be $A$-dualizable 
(respectively, $\cS$-dualizable) if there is an
  $A^{op}$-module ${}^\vee M$ (respectively $ M^\vee$) together with morphisms
$$
\xymatrix{
ev^L:  M \ot_\cS{}^{\vee} M  \ar[r] & A \mbox{ in }\cC &
  coev^L:1_\cS\ar[r] &  {}^{\vee}M\ot_A M\mbox{ in }\cS
  }
  $$
  (respectively,
  $$
  \xymatrix{
  ev^R: M^\vee\ot_A  M\ar[r] & 1_\cS\mbox{ in }\cS & 
  coev^R:A \ar[r] &  M\ot_\cS  M^\vee\mbox{ in }\cC)
  }$$ satisfying the usual
  identities.

\end{definition}

\begin{remark}\label{infty,2}
The notion of dualizability for $A$-modules is perhaps more naturally seen as usual dualizability for bimodules for algebra objects $A$ and $B$ (in the case $B=1_\cS$), under the convolution product on bimodules. 

Recall (see \cite{jacob TFT}) that by considering monoidal
categories as 2-categories with a single object, we can understand
duality for objects in a monoidal category as a special case of the
notion of adjoints in a $2$-category. The above notion of duality for
$A$-modules (or bimodules) has a similar interpretation. Namely, we consider an
$A$-module $ M\in \cC$ as a bimodule for $A$ and the unit $1_S$, and
thus as a morphism $A\to 1_S$ in the Morita $(\oo,2)$-category
$Alg_{(1)}(\cS)$ of algebra objects in $\cS$. In this interpretation, $ M$ is $A$- (respectively,
$\cS$-) dualizable precisely when $ M$ admits a left (respectively,
right) adjoint. 
\end{remark}

We record here the following elementary compatibility of our previous definitions.

\begin{lemma}\label{duals compatible}
An $A$-module $M\in \Mod_A(\cS)$ is $\cS$-dualizable if and only if the underlying object $M\in \cS$ is dualizable.
In this case, the $\cS$-dual $M^\vee\in \Mod_{A^{op}}(\cS)$ has underlying object the dual $M'\in \cS$, compatibly with evaluation and coevaluation maps.
\end{lemma}

\begin{proof}

Suppose $M$ is $\cS$-dualizable with $\cS$-dual $M^\vee$. Then we can extend its evaluation morphism
$$
\xymatrix{
ev:M\otimes_\cS M^\vee \ar[r] & M\otimes_A M^\vee \ar[r]^-{ev^R}  & 1_\cS
}
$$
where the first map is the natural quotient, and its coevaluation morphism
$$
\xymatrix{
coev:1_\cS \ar[r] & A \ar[r]^-{coev^R} & M^\vee \otimes_\cS M
}
$$
where the first map is the unit of the algebra $A$.
It is straightforward to check that these exhibit $M^\vee$ as the plain dual of $M$.

Conversely, suppose $M$ is dualizable with dual $M'$. There is a canonical identification of algebras
$$
\End_\cS(M') \simeq \End_\cS(M)^{op},
$$
and so $M'$ inherits a natural $A^{op}$-module structure. By construction, the evaluation $ev$ descends to provide the evaluation $ev^R$. Likewise, by adjunction, the coevaluation $coev$ extends to provide the $A$-bilinear coevaluation $coev^R$.
\end{proof}

\subsection{Centers and traces}\label{defining centers}

We briefly review the theory of Hochschild (co)homology of associative
algebra objects in closed symmetric monoidal $\infty$-categories,
following \cite{BFN} to which we refer for more details.  This is an $\oo$-categorical
version of the approach to topological Hochschild homology
developed in \cite{EKMM,shipley THH}.

Let $\cS$ be a closed symmetric monoidal $\infty$-category, and $A\in \cS$ an associative algebra object. 
Then we have internal hom objects
\cite[6.1]{HA}, and given $A\ot A^{op}$-modules $M,N$, we can consider the $A\ot A^{op}$-linear
morphism object $\intHom_{A\ot A^{op}}(M, N)\in \cS$.
 Likewise, given left and right $A\ot A^{op}$ modules
$M,N\in\cS$ we have a pairing $M\ot_{A\ot A^{op}} N\in \cS$ defined by the two-sided
bar construction  \cite[4.3.4]{HA} over $A\ot A^{op}$.

\begin{definition} \label{def center}\
Let $A$ be an associative algebra object in a closed symmetric monoidal
$\oo$-category~$\cS$, and $B$ an $A$-bimodule in $\cS$.

\begin{enumerate}

\item The Hochschild cohomology (or derived center)
  $\Z(A)=\hh^*(A)\in \cS$ is the endomorphism object $\intEnd_{A\ot
  A^{op}}(A)$ of $A$ as an $A$-bimodule. More generally, the Hochschild cohomology of the bimodule $B$ is
  $$\hh^*(A,B)=\Hom_{A\ot A^{op}}(A,B).$$

\item The Hochschild homology (or derived trace)
  $\Tr(A)=\hh_*(A)\in \cS$ is the pairing object ${A\ot_{A\ot A^{op}}
  A}$ of $A$ with itself as an $A$-bimodule. More generally, the Hochschild homology of the bimodule $B$ is
  $$\hh_*(A,B)=A\ot_{A\ot A^{op}}B.$$

\end{enumerate} 
\end{definition}

\begin{example}
Taking $B=M_A\ot {}_A N$ for a left $A$-module ${}_A N$ and a right $A$-module $M_A$, we find the relative tensor product $$\hh_*(A,M_A\ot {}_A N)=M\ot_A N,$$
while taking $B=\Hom({}_A M,{}_A N)$ the Hom object in $\cS$ between two left $A$-modules $M,N$, we find the $A$-module Hom
$$\hh^*(A,\Hom({}_A M,{}_A N))=\Hom_A(M,N).$$
\end{example}

In general, the center $\Z(A)$ is again an associative algebra object in
$\cS$.  Furthermore, $\Z(A)$ comes with a canonical central morphism
$$
\xymatrix{
\fz:\Z(A)
\ar[r] &
A 
&&
F\ar@{|->}[r]
& 
F(1_A)
}
$$ while the trace $\Tr(A)$ comes with a canonical trace morphism
$$
\xymatrix{
\ftr:A 
\ar[r] &
\Tr(A)
}
$$ 
coequalizing left and right multiplication.


\sssn{Cyclic bar construction}\label{hochschild complexes}
For an associative algebra object $A$ in a symmetric monoidal $\infty$-category
$\cS$, there is a natural simplicial object $\mathbf{N}^{cyc}_*(A)$ 
and cosimplicial object $\mathbf{N}_{cyc}^*(A)$ such that the geometric realization
 $\colim \mathbf{N}^{cyc}_*(A)$ calculates $\hh_*(A)=\Tr(A)$ and the
totalization  $\lim \mathbf{N}_{cyc}^*(A)$ calculates $\hh^*(A)=\Z(A)$. 

We construct the simplicial object $\mathbf{N}^{cyc}_*(A)$ and
cosimplicial object $\mathbf{N}_{cyc}^*(A)$ as follows. The
$A$-bimodule $A$ has a canonical free simplicial resolution $C_*(A)$ as an
$A$-bimodule with terms 
$$
C_{n}(A)\simeq A^{\ot n+2}
$$ 
It is defined using
the usual formalism of cotriple resolutions as in \cite{BFN}. The face maps come from the multiplication on $A$,
and the degeneracies from the unit of $A$. The $A$-bimodule structure is the obvious action on the left and right extreme factors.

Now recall that the trace $\hh_*(A,B)$ is defined by the pairing
$A\ot_{A\ot A^{op}}B$. Since the tensor product commutes with colimits, in particular geometric realizations, we calculate
$$
\hh_*(A,B) = A\ot_{A\ot A^{op}}B \simeq |C_*(A)|\ot_{A\ot A^{op}}B
\simeq |B\ot_{A\ot A^{op}} C_*(A)|.
$$
Thus the geometric realization of $ C_*(A)\ot_{A\ot A^{op}} B$
calculates the trace $\hh_*(A,B)$.

We write $\mathbf{N}^{cyc}_*(A)$
for the simplicial object $A\ot_{A\ot A^{op}} C_*(A)$ and refer to it as the  Hochschild simplicial object.
We can evaluate the terms of the Hochschild simplicial object
$$
\mathbf{N}^{cyc}_n(A) = A\ot_{A\ot A^{op}} C_n(A)
\simeq
A\ot_{A\ot A^{op}} A^{\ot n+2} \simeq A^{\ot n+1}
$$ 
In particular, there are equivalences $\mathbf{N}^{cyc}_0(A) \simeq A$ and $\mathbf{N}^{cyc}_1(A) \simeq A\ot A$, and the two simplicial maps $A\ot A \ra A$ are 
the multiplication and the opposite multiplication of $A$. 

Similarly, recall 
that the center $\hh^*(A,B)$ is defined by the hom object
$\intHom_{A\ot A^{op}}(A,B)$. Since morphisms take colimits in the domain to limits, 
in particular geometric realizations to totalizations, we calculate
$$
\hh^*(A,B) = \intHom_{A\ot A^{op}}(A,B)
\simeq  \intHom_{A\ot A^{ op}}(|C_*(A)|, B)
\simeq \lim \intHom_{A\ot A^{ op}}(C_*(A), B)
$$
Thus the totalization of $\intHom_{A\ot A^{ op}}(C_*(A), B)$
calculates the center $\hh^*(A,B)$.

We write $\mathbf{N}_{cyc}^*(A)$
for the cosimplicial object $\intHom_{A\ot A^{ op}}(C_*(A), A)$
 and refer to it as the Hochschild cosimplicial object. 
As before, we can evaluate the terms of the
 Hochschild cosimplicial object
$$
\mathbf{N}_{cyc}^n(A) = \intHom_{A\ot A^{ op}}(C_n(A), A)
\simeq
\intHom_{A\ot A^{ op}}(A^{\ot n+2}, A)
\simeq
\intHom(A^{\ot n}, A). 
$$ 
In particular, there are equivalences $\mathbf{N}_{cyc}^0(A) \simeq A$ and 
$\mathbf{N}_{cyc}^1(A) \simeq \intHom(A, A)$, and the two cosimplicial maps $A\to \intHom(A, A)$ are induced by
the left and right multiplication of $A$. 

\sssn{Transposed Hochschild object}\label{transposed HH}
Let us now suppose that $A\in \cS$ is dualizable, with dual $A'$. 
In that case the simplicial object $C_*(A)$ defines a dual cosimplicial object $C_*(A)^t$ with cosimplices
$$
C_{n}(A)^t\simeq A'^{\ot n+2}
$$ 
and face and degeneracies given as the transposes of those of $C_*(A)$.
In this case, we can describe the center of a bimodule as a colimit
\begin{eqnarray*}
\hh^*(A,B) &\simeq &  \lim \intHom_{A\ot A^{ op}}(C_*(A), B)\\
&\simeq & | C_*(A)' \ot_{A\ot A^{op}} B|
\end{eqnarray*}


\section{Rigidity and dualizability}\label{rigid categories}

In this section, we develop some general properties of monoidal categories with duals. 

See~\cite[Section 6]{DGcat} for some related arguments 
under different hypotheses; for our applications
we need to remove the assumption that the unit is compact.

\subsection{Semi-rigid categories} 
We continue here in the general setting $\St_\CC$ of stable presentable $\oo$-categories. Given a monoidal category $A\in Alg(\St_\CC)$, its 
multiplication map $\mu:A\otimes A\to A$ can be regarded as a map of $A$-bimodules, or equivalently $A^e=A\ot A^{op}$ modules, where
the left (respectively, right) $A$-action on $A\otimes A$ is via the first (respectively, second) factor.

We will be interested in monoidal categories with enough duality in the following sense:

\begin{definition}
A monoidal category $A\in Alg(\St_\CC)$ is said to be 
{\em semi-rigid} if $A$ is compactly generated by objects that are both right and left dualizable.
\end{definition}

The definition is a variant of the standard notion of rigid tensor category, for which the 
following $\infty$-refinement can be found in \cite{DGcat}:

\begin{definition}  
A monoidal category $A\in  Alg(\St_\CC)$ is said to be {\em rigid} if
\begin{enumerate}
\item[$\bullet$] The multiplication map $\mu:A\otimes A\to A$ is quasi-proper.
\item[$\bullet$] The natural lax $A^e$-linearity of the right adjoint $\mu^r$ is strict.
\item[$\bullet$] The unit $1_A\in A$ is compact.
\end{enumerate}
\end{definition}

The terminology ``semi-rigid"
is justified by the following assertion. 

\begin{prop} A semi-rigid category with compact unit is rigid. Conversely, a compactly generated rigid category is semi-rigid.
\end{prop}

\begin{proof}
The first part follows from Lemmas \ref{quasiproper action} and \ref{perfect adjoints} below, 
the second is \cite[6.1.1]{DGcat}.
\end{proof}

\begin{lemma}\label{quasiproper action}
If $A\in  Alg(\St_\CC)$ is a semi-rigid category, the multiplication morphism $\mu:A\otimes A\to A$ is quasi-proper,
and so its right adjoint $\mu^r:A\to A\otimes A$ is continuous. More generally, for any $A$-module $M$, the 
action map $A\ot M\to M$ is quasi-proper.
The unit morphism $e:\Vect_\CC\to A$ is quasi-proper if and only if in addition $A$ is rigid.
\end{lemma}

\begin{proof}
For the first assertion, let $a\in A$ be compact (hence dualizable), $m\in M$ compact, and $\{n_i\}$ an arbitrary small diagram in $M$. 
Then we calculate using Lemma \ref{internal}:
\begin{eqnarray*}
\Hom_{{M}}(a\star m , \colim_i n_i))
&\simeq&  \Hom_{{M}}(m, {}^\vee a\star (\colim_i n_i))\\
&\simeq&  \Hom_{{M}}(m, \colim_i ({}^\vee a\star n_i))\\
&\simeq&  {\lim}_i \Hom_{{M}}(m, {}^\vee a\star n_i)\\
&\simeq&  {\lim}_i \Hom_{{M}}(a\star m, n_i)
\end{eqnarray*}
Hence the action of $A$ on $M$ is quasi-proper, and so by Lemma~\ref{quasi-proper functors}, its right adjoint continuous.

Since $\Hom_A(1_A, -):A\to \Vect_\CC$ is the right adjoint to $e:\Vect_\CC\to A$, the second assertion also follows from Lemma~\ref{quasi-proper functors}.
\end{proof}

\subsubsection{Linearity of adjoints}

An important technical aspect of working with semi-rigid categories is that adjoints of linear functors
between modules are canonically linear (compare \cite[Corollary 6.2.4]{DGcat}):

\begin{lemma}\label{perfect adjoints} Let $A \in  Alg(\St_\CC)$ be a semi-rigid monoidal category.
 
 Let $\cC,\cD \in \St_\CC$ be compactly generated $A$-modules, and 
  $F:\cC\to\D$ a quasi-proper (respectively, accessible
  limit-preserving) $A$-linear functor. 
  Then the right (respectively, left) adjoint of $F$ is canonically 
  $A$-linear and continuous as well. 
  \end{lemma}

\begin{proof}
 If $F:\cC\longleftrightarrow \cD: G$ form an adjoint pair, and $F$ is
equipped with an $A$-linear structure, then $G$ automatically commutes with the $A$-action in a lax sense, i.e.~up to a natural transformation.
Thus we need to check that these natural transformations are isomorphisms, which we do by writing objects as colimits of compact generators:
\begin{eqnarray*}
\Hom_{{\cC}}(\colim_i c_i, G( \colim_j a_j\star d))
&\simeq& \Hom_{{\cD}}(F(\colim_i c_i), \colim_J a_j\star d)\\
&\simeq& \Hom_{{\cD}}(\colim_i F(c_i), \colim_j a_j\star d)\\
&\simeq& \rm{lim}_i \Hom_{{\cD}}(F(c_i), \colim_j a_j\star d)\\
&\simeq& \colim_j {\rm{lim}}_i \Hom_{{\cD}}(F(c_i),  a_j\star d)\\
&\simeq& \colim_j {\rm{lim}}_i \Hom_{{\cD}}( a_j^\vee \star F(c_i), d)\\
&\simeq& \colim_j {\rm{lim}}_i \Hom_{{\cD}}(F( a_j^\vee \star c_i), d)\\
&\simeq& \colim_j {\rm{lim}}_i \Hom_{{\cC}}( a_j^\vee \star c_i, G(d))\\
&\simeq& \colim_j {\rm{lim}}_i \Hom_{{\cC}}(c_i,  a_j\star G(d))\\
&\simeq& \Hom_{{\cC}}(\colim_i c_i, \colim_j a_j\star G(d))
\end{eqnarray*}

 The argument for left
adjoints of $A$-linear functors is identical.
\end{proof}

\begin{cor}
 Let $A \in  Alg(\St_\CC)$ be a semi-rigid monoidal category.

Passing to Ind-categories and passing to right adjoints induce canonical equivalences between 
the following three $\oo$-categories:
\begin{enumerate}
\item[$\bullet$] the category of $A_c$-modules in small idempotent-complete stable categories with morphisms exact functors;
\item[$\bullet$] the category of $A$-modules in compactly generated stable categories with morphisms quasi-proper functors;
\item[$\bullet$] the opposite of the category of $A$-modules in compactly generated stable categories with morphisms accessible continuous and cocontinuous functors. 
\end{enumerate}
\end{cor}

\begin{remark}
This corollary allows us to work either with plain stable presentable categories $\St$ or in the $\CC$-linear setting in calculating limits and colimits 
of categories.
\end{remark}

\subsubsection{Adjoints} 
It follows from the above discussion that a semi-rigid monoidal category $A\in Alg(\St_\CC)$ defines a comonoidal
category $A^r\in Coalg(\St_\CC)$ with the same underlying category $A^r=A$ and comultiplication $\mu^r$ the right adjoint of
the product $\mu$ on $A$. Given an $A$-module $M$, the diagram defining its module structure defines, by taking right adjoints, a diagram in the opposite category
exhibiting $M$ as a comodule for $A^r$. Moreover, this comodule structure is in fact functorial for all morphisms of $A$-modules. To see this, note
that given an action $a_M:A\ot M\to M$, the coaction of $A^r$ can be described as the composition
$$\xymatrix{a_M^r: M\ar[r]^-{\sim} & A\ot_A M \ar[r]^-{\mu^r\ot \Id} & A\ot A\ot_A M\ar[r]^-{\sim} & A\ot M }$$
where we use semi-rigidity to ensure that $\mu^r$ is a morphism of right $A$-modules.
This construction evidently commutes with all morphisms of $A$-modules. 
In other words, we have an equivalence 
$$
\Mod_A(\St_\CC) \simeq \Comod_{A^r}(\St_\CC)
$$ 
commuting with the natural forgetful functors to $\St_\CC$.

\subsection{Transposes}
A semi-rigid monoidal category $A\in Alg(St_\CC)$, being compactly generated, is dualizable with dual $A'$. 
The  coevaluation functor, 
$$
\xymatrix{
\eta:\Vect_\CC\ar[r] & A'\ot A
}$$ sends the unit $\CC$ to the Hom 
bimodule $[\Hom]\in A'\ot A$ which is determined by 
$$\Hom_{A'\ot A}(m\ot n, [\Hom])=\Hom_A(m,n)$$
on compact objects $m\in A_c^{\catop}, n\in \ot A_c$.

The dual $A'$ inherits a canonical coalgebra structure from the algebra structure on $A$ by taking transposes of
all multiplication maps. 
An $A$-module $M$,  with action map $a_M:A\ot M\to M$, inherits a comodule structure over $A'$, with coaction map given by the transpose
$$\xymatrix{a_M^t: M\ar[r]^-{\eta\ot \Id} & A'\ot A\ot M \ar[r]^-{\Id\ot a_M} & A'\ot M}$$
It is evident from this construction that the comodule structure is functorial for all
$A$-linear functors $M\to N$. In other words, we have an equivalence 
$$\Mod_A(\St_\CC) \simeq \Comod_{A'}(\St_\CC)$$ 
commuting with the natural forgetful functors to $\St_\CC$.

The left and right dualities on compact objects
$$
\xymatrix{
L(a) = {}^\vee a,&  R(a) =  a^\vee, &
\mbox{for $a\in A$ compact},
}
$$ 
canonically extend by continuity to
equivalences 
$$
\xymatrix{
L, R: A\ar[r]^-\sim & A'
}
$$
Let $LL,RR:A\to A$ denote the inverse monoidal automorphisms of $A$ given by taking double duals.
Moreover, the identities
$$\Hom_A({}^\vee(x\star m \star y), n)\simeq \Hom_A({}^\vee m, {}^{\vee\vee}y \star n\star x), $$
$$\Hom_A((x\star m \star y)^\vee, n)\simeq \Hom_A(m^\vee, y \star n\star x^{\vee\vee})$$
show that $L$ is a naturally a morphism of right $A$-modules and $R$ of left $A$-modules, while $L$ intertwines
the left action of $A$ with the action twisted by the monoidal automorphism $LL:A\to A$, and likewise $R$ intertwines
the right action of $A$ with the action twisted by the monoidal automorphism $RR:A\to A$.
In other words, we have isomorphisms of $A$-bimodules
$$\xymatrix{{}_{RR}A\ar[r]^-{L} & A'&\ar[l]_-{R}  A_{LL}}$$
where we twist the bimodule $A$ on one side or the other by the inverse monoidal automorphisms $LL,RR$.

\begin{definition}  A {\em pivotal structure} on a semi-rigid monoidal category $A \in  Alg(\St_\CC)$ is a monoidal identification of the automorphism
$LL$ (or equivalently, of its inverse $RR$) with the identity of $A$.
\end{definition}

It follows that a pivotal structure is equivalent to a lift of either of the functors
$L,R:A\to A'$ to a morphism of $A$-bimodules  (or to a monoidal identification of left and right duals in $A$).

\begin{example}
For any semi-rigid monoidal category $A\in  Alg(\St_\CC)$, the semi-rigid monoidal category $A^e=A\ot A^{op}$ has a canonical pivotal structure, deduced from its
symmetric structure $\xymatrix{A^e\ar[r]^-{\sim} & (A^e)^{op}}$ exchanging left and right duals. 
\end{example}

\begin{remark}[The canonical trace]
It is enlightening (though not necessary for the rest of this paper) to understand the notion of a pivotal structure in terms of traces.

Suppose $A\in  Alg(\St_\CC)$ is a semi-rigid monoidal category. Consider the unit morphism $u:\Vect_\CC\to A$ and its transpose
$$
\xymatrix{ 
\tau:=u^{tr}:A\ar[r]^-{R} & A' \ar[r]^-{u'} & \Vect_\CC' \ar[r]^-{R^{-1}} & \Vect_\CC
    }$$ 

If $a\in A$ is compact, then we have 
$$
\tau(a) \simeq \Hom_A(1_A, a)
$$
and in general $\tau$ is the continuous extension of this assignment (which will agree with the Hom functional
if and only if $A$ is rigid). 

Given compact objects
$a, b \in A$, we have canonical identifications
\begin{equation*}
\tau(a \star b)=
 \Hom_A(1_A, a\star b)
\simeq    \Hom_A((b^\vee)^\vee, a)
 = \tau((b^\vee)^\vee \star a)
\end{equation*}

It follows that a pivotal structure on $A$, i.e., a monoidal identification $RR\simeq \Id$, gives rise to a trace structure on the
functional $\tau$, i.e., a functorial identification $\tau(a\star b)\simeq \tau(b\star a)$ plus higher compatibilities. More precisely,
a pivotal structure gives rise to a factorization $$\xymatrix{A\ar[r]^-{\mathfrak{tr}} \ar[dr]_-{\tau} & \Tr(A) \ar[d]\\ & \cS}$$ and is in fact equivalent to such a
factorization.
\end{remark}

\subsubsection{Adjoints are transposes}
We may now identify the two coalgebras $(A,\mu^r)$ and $(A',\mu^t)$ obtained from a semi-rigid monoidal category:

\begin{prop}   Let $A \in  Alg(\St_\CC)$ be a semi-rigid monoidal category.

There is a canonical equivalence of comonoidal categories $(A,\mu^r)\simeq (A',\mu^t)$ given on the level
of objects by taking right duals, i.e., by the equivalence $R:A\to A'$.

In particular, for an $A$-module $M$ with action map $a_M:A\ot M\to M$, the right adjoint 
$a_M^r:M\to A\ot M$ and the transpose $a_M^t: M\to A'\ot M$ are related by the duality $R$ in the sense
that the following diagram canonically commutes:
$$\xymatrix{ M \ar[dr]_-{a_M^r} \ar[r]^-{a_M^t} & A' \ot M \ar[d]^-{R\ot \Id}\\
& A\ot M}$$
\end{prop}

\begin{proof}
The first assertion follows from the identifications 
$$\Comod_{A^r}(\St_\CC)\simeq \Mod_A(\St_\CC) \simeq \Comod_{A'}(\St_\CC)
$$ of categories comonadic over $\St_\CC$, whence the coalgebras $A^r$ and $A'$ representing 
the corresponding comonads are naturally identified.

To verify that the identification is given by right duals, let us check by hand the second part of the proposition.
Note that the action $a_M$ is quasi-proper by Lemma \ref{quasiproper action}.
Since $A$ generates the category of $A$-modules by colimits, it suffices to establish the claim
in the universal case $M=A$, where $a=a_A$ is the left action of $A$ on itself.
Thus we have to identify the map $$
\xymatrix{\mu^r:A\ar[r] &A\ot A
}$$ with the map
$$\xymatrix{A\ar[r]^-{a^t} & A'\ot A \ar[r]^{R\ot\Id} & A\ot A.}$$
Both are maps of right $A$-modules, hence it suffices to identify the images of the unit $$
\xymatrix{
\eta_A:\Vect_\CC\ar[r] & A\ar[r] & A\ot A
}$$
under the two maps. For the first, it is characterized on compact objects by
\begin{eqnarray*}
\Hom_A(m\ot n, \mu^r(1_A))&\simeq& \Hom_A(m\star n, 1_A)\\
&\simeq& \Hom_A(m, n^\vee)
\end{eqnarray*}
For the second, the composition
$$
\xymatrix{\Vect_\CC\ar[r]& A\ar[r]& A'\ot A\ar[r] & A\ot A}$$
sends the unit to $(R\ot \Id)([\Hom])\in A\ot A$, which is characterized on compact objects by
\begin{eqnarray*}
\Hom_{A\ot A}(m\ot n, (R\ot \Id)([\Hom])) &\simeq& \Hom_A({}^\vee m, n)\\
 &\simeq & \Hom_A(m, n^\vee)
\end{eqnarray*}
\end{proof}

\subsection{Bar constructions} Recall that $C_*(A)$ denotes
the canonical simplicial resolution of $A$ as an $A$-bimodule, i.e., the monadic bar construction
for the algebra $A$.
Let $C_*(A)^{tr}$ denote the cosimplicial object obtained
by passing to transposes, i.e., the comonadic cobar construction for the coalgebra $(A',\mu^t)$. Likewise, let
$C_*(A)^r$ denote the cosimplicial object obtained by passing to right
adjoints, i.e., the comonadic cobar construction for $(A^r,\mu^r)$. Note however that unless $A$ is rigid, the right adjoint
to the unit morphism is not continuous, hence this cosimplicial object is not a diagram of continuous functors.  We use the term
semi-cosimplicial object to refer to the diagram object obtained from
such a cosimplicial object by forgetting its codegeneracies, see
\cite[Notation 6.5.3.6]{topos}. By \cite[Lemma 6.5.3.7]{topos}, the
inclusion of the semisimplical $\infty$-category in the simplicial one
is cofinal, so that we may calculate totalizations of cosimplicial
objects as limits over the underlying semi-cosimplicial objects. The underlying semi-cosimplicial object of $C_*(A)^r$ is
indeed a diagram of continuous functors.
The following is an immediate corollary of the above proposition:

\begin{cor}\label{reversing arrows}
If $A\in  Alg(\St_\CC)$ is a semi-rigid monoidal category, then the semi-cosimplicial objects underlying $C_*(A)^r$ and $C_*(A)^{tr}$
are canonically equivalent. (If in addition $A$ is rigid, the cosimplicial objects 
$C_*(A)^r$ and $C_*(A)^{tr}$
are canonically equivalent.) In fact, there is a natural equivalence of diagrams of $A$-bimodules
$$\xymatrix{ C_*(A)^r_{LL} \ar[r]^-{R} &C_*(A)^{tr}}$$
where the right $A$-action on $C_*(A)^r$ is twisted by the monoidal automorphism $LL$.

\end{cor}

Recall from Section \ref{transposed HH} that the cosimplicial diagram $C_*(A)^{tr}$ can be used to calculate Hochschild cohomology
of $A$-bimodules. The cosimplicial object $C_*(A)^r$ can likewise be used to calculate Hochschild homology, so that Corollary 
\ref{reversing arrows} results in a general identification of traces and centers:

\begin{prop}\label{reversing hh} Let $A\in  Alg(\St_\CC)$ be a semi-rigid monoidal category.
\begin{enumerate}
\item For any $A$-bimodule, we have an equivalence between a twist of the center and the trace: 
$$\hh_*(A,B)\simeq \hh^*(A,{}_{LL} B).$$

\item In particular, for $M$ a dualizable left $A$-module with dual $M'$ and $N$ an arbitrary left $A$-module,
there is a canonical equivalence $$\Hom_A(M,{}_{LL}N) \simeq M' \ot_A N$$ between a twist of the relative $\Hom$ and the relative tensor
product of the dual of $M$ with $N$. 
\end{enumerate}
\end{prop}

\begin{proof}
Recall the description of Hochschild homology via the bar resolution, $$
\hh_*(A,B) = A\ot_{A\ot A^{op}}B \simeq \colim C_*(A)\ot_{A\ot A^{op}}B.$$
By Section \ref{reversing diagrams}, this colimit can be identified with the limit over
the corresponding cosimplicial diagram, obtained by taking right adjoints (note that the simplicial object $C_*(A)\ot_{A\ot A^{op}} B$, the relative tensor product
of a diagram of continuous functors by $B$, is a diagram
of continuous functors). This adjoint cosimplicial diagram is naturally identified with $C_*(A)^r\ot_{A\ot A^{op}} B$, i.e., can be calculated
by taking right adjoints in $C_*(A)$.
Hence by Corollary \ref{reversing arrows}, we find
\begin{eqnarray*}
\hh_*(A,B) & \simeq& \colim C_*(A)\ot_{A\ot A^{op}}B\\
&\simeq& \lim C_*(A)^r \ot_{A\ot A^{op}} B\\
&\simeq& \lim C_*(A)^t_{RR} \ot_{A\ot A^{op}} B\\
&\simeq& \lim C_*(A)^t\ot_{A\ot A^{op}} {}_{LL} B\\
&\simeq& \hh^*(A,{}_{LL} B).
\end{eqnarray*}

The second assertion follows immediately from the special case $B=N\ot M'$. 

\end{proof}

\begin{prop} \label{bimodule dual}
For $A$ semi-rigid and $M$ an $A$-module,
$M$ is $A$-dualizable if and only if $M$ is dualizable in $\St_\CC$. In this case, the $A$-dual is canonically identified with the twisted naive dual $M'_{RR}$ as right $A$-modules.

In particular, the canonical pivotal structure on $A^e= A\otimes A^{op}$ induces an identification of the $A^e$-dual and $\St_\CC$-duals of $A$ (or any $A$-bimodule).
\end{prop}

\begin{proof}
Suppose $M$ is dualizable in $\cS$. 
We first apply Proposition \ref{reversing hh}  and find
\begin{eqnarray*}
\Hom_{A}(M,A)&\simeq& M' \ot_A {}_{LL} A\\
 &\simeq& M'_{RR}. 
 \end{eqnarray*}

On the other hand 
 \begin{eqnarray*}
 M'_{RR} \otimes_{A} M &\simeq& M' \otimes_{A} {}_{LL}M\\
 &\simeq& \Hom_{A}(M,M).
 \end{eqnarray*}

We can now realize $M'_{RR}$ as the $A$-dual of $M$ by taking the canonical evaluation
and coevaluation  maps
$$
\xymatrix{
M\ot  M'_{RR} \simeq M\otimes \Hom_{A}(M,A)\ar[r]^-{ev} &  A
&
\Vect_\CC \ar[r] &  \Hom_{A}(M,M)\simeq M'_{RR}\ot_{A} M
}
$$

Similarly, if $M$ is $A$-dualizable with dual $\wt{M}$, we can explicitly exhibit a duality in $\St_\CC$ between
$M$ and $M'=\wt{M}_{LL}$.
\end{proof}


\subsection{Dualizability}
\label{full dualizability}

We now interpret the results above as establishing the 2-dualizability of semi-rigid categories as objects in the symmetric monoidal
$(\infty,2)$-category $Alg_{(1)}(\St_\CC)$, the Morita 2-category of monoidal categories~\cite[Section 4.1]{jacob TFT}. This interpretation will not require any technical notions of $(\infty,2)$-category theory --- 
we recall the notion of 2-dualizability as it specializes in our setting to specific identities formulated for monoidal $\oo$-categories.

A Morita morphism between two (stable presentable) monoidal $\oo$-categories $A, B$ is an $A^{op}\otimes B$-module category.
The evaluation and coevaluation Morita morphisms for a monoidal category $A$ are the Morita morphisms
$$
\xymatrix{
\Vect_\CC \ar[r]^-{A} & A\ot A^{op} \ar[r]^-{A} & \Vect_\CC
}
$$
given by the object $A$ itself in an obvious way.

\begin{definition}\label{jacob defs}
(1) \cite[Proposition 4.2.3]{jacob TFT} A 2-dualizable object of
  $Alg_{(1)}(\St_\CC)$ is a monoidal category $A$ such that the evaluation Morita morphism $ev_A:A\ot
  A^{op}\to \cS$ has both a right and a left adjoint.

(2) \cite[Definition 4.2.6]{jacob TFT} A Calabi-Yau structure on a monoidal category $A\in Alg_{(1)}(\St_\CC)$
is an $S^1$-invariant functional $\Tr(A)=A\ot_{A\ot
  A^{op}} A\to \cS$ which is the counit of an adjunction between
$ev_A$ and $coev_A$. 
\end{definition} 

In this definition, an adjoint for a Morita morphism is given by a (left or right) dual Morita morphism (as in Remark~\ref{infty,2})
together with unit and counit morphisms satisfying the standard identities.
The $S^1=SO(2)$ action on $\Tr(A)$ is an $\infty$-categorical version of the Connes cyclic structure on Hochschild homology, and is 
provided by the one-dimensional Cobordism Hypothesis, thanks to the identification of $\Tr(A)$ with 
the dimension of the dualizable object $A$ in the underlying Morita $(\infty,1)$-category. 
It can be described explicitly via topological chiral homology, see the proof of Theorem~\ref{CY Hecke}.

We will also need a natural weakening of the notion of Calabi-Yau structure.
Namely, a 2-dualizable monoidal category $A$ admits a canonical automorphism $S_A$, the Serre automorphism, with the property
$$ev^R_A = (S_A\ot \Id_{A^{op}})\circ coev_A$$
(see \cite[Proposition 4.2.3]{jacob TFT}). It can be interpreted as a monodromy of the object $A$ around the $SO(2)$-action, i.e., the first obstruction
to $SO(2)$-invariance of $A$, or as encoding the action of the free group on the circle $\Omega S^2=\Omega\Sigma SO(2)$ on $A$.

\begin{definition}
A weakly Calabi-Yau monoidal category is a 2-dualizable monoidal category with a trivialization of the
Serre automorphism $S_A\simeq \Id_A$, i.e., an isomorphism of $A$-bimodules $S_A\simeq A$. 
\end{definition}

Given such a trivialization, one can define a trace functional 
$$
\xymatrix{
A\ar[r]^-{\mathfrak{tr}} \ar[dr]_-{\tau} &\Tr(A)\simeq ev_A \circ coev_A \simeq ev_A\circ ev^R_A \ar[d] \\&  \Vect_\CC}
$$
with the vertical map the counit of the adjunction.
A Calabi-Yau structure is then a factorization of this trace through the fixed points of
$S^1$ on $\Tr(A)$.

The notion of 2-dualizability reproduces the two senses of duality
for $A$ as an $A^e$-module appearing in Definition~\ref{duality for
  modules}.

%
%

\begin{lemma}\label{separable} 
(1) A monoidal category $A$ is $A^e$-dualizable if and only if the evaluation Morita morphism $ev_A:A^e\to \Vect_\CC$ admits a left adjoint. 
In this case, the left adjoint is given by the bimodule dual~$ A^! \in A^e\on{-mod}$.

(2)  A monoidal category $A$ is $\St_\CC$-dualizable
if and only if the evaluation Morita morphism $ev_A:A^e\to \Vect_\CC$ admits a colimit-preserving right adjoint.
In this case, the right adjoint is given by the $\St_\CC$-dual $A'\in A^e\on{-mod}$.

(3) The Serre automorphism $S_A$ of a 2-dualizable category $A$ is represented (as Morita endomorphism of $A$) by the $A$-bimodule $A'$.
\end{lemma}

\begin{proof}
The assertions are immediate from the definitions.
The bimodule giving the left
adjoint is the bimodule dual $ A^!$.
The bimodule giving the right
adjoint is the plain dual $A'$. 
Hence the discrepancy $S_A$ between the right adjoint and the coevaluation ${}_{\Vect_\CC}A_{A^e}$ is also given by the bimodule $A'$.
%
%
%
%
%
%
%
%
%
%
\end{proof}

Now our prior results admit the following reinterpretation.

\begin{theorem}\label{perfect dualizable} Let $A\in  Alg(\St_\CC)$ be a 
semi-rigid category.  Then $A$ is 2-dualizable. 

A pivotal structure on $A$ gives rise to a weak Calabi-Yau structure on $A$.
\end{theorem}

\begin{proof}
Proposition~\ref{jacob duality} and Lemma~\ref{duals compatible} imply $A$ is $\St_\CC$-dualizable.
Proposition~\ref{bimodule dual} implies $A$ is $A^e$-dualizable, with $A^e$ and $\cS$-dual identified. Hence Lemma~\ref{separable}
implies $A$ is fully dualizable. Moreover a pivotal structure on $A$ is equivalent to a particular identification $S_A=A'\simeq A$ of bimodules, as required
to make $A$ weakly Calabi-Yau.
\end{proof}


\sn{Preliminaries on $\D$-modules}\label{D modules}


This section is devoted to a review of the theory of $\D$-modules on schemes and stacks following \cite{finiteness}. As  
explained in \cite{finiteness}, while this material is well-known at the level of triangulated categories (standard 
references include Bernstein's lecture notes~\cite{bernstein}
and the books by Borel~\cite{borel} and Kashiwara~\cite{kashiwara} for $\D$-modules on varieties, and Bernstein-Lunts' book~\cite{BL} 
and Beilinson-Drinfeld's manuscript~\cite[Chapters 1 and 7]{BD} for $\D$-modules on stacks), no comprehensive account seems to exist yet at the level of
 dg categories (though the aspects of the theory that concern pullbacks are developed in \cite{dennisnick}).


\ssn{$\D$-modules on schemes}
Throughout what follows, 
by a scheme, we will always mean a quasi-compact, separated derived scheme over $\CC$.

To any scheme $X$, there is attached a dg category  $\D(X)$ of $\D$-modules on $X$, defined by taking
 ind-coherent sheaves on the de Rham stack of $X$.
The category $\D(X)$ is insensitive to the derived or
non-reduced structure of $X$, so one may assume $X$ is a reduced classical scheme for concreteness. 
When $X$ is smooth, $\D(X)$ can be canonically identified with the traditional dg derived category of 
right $\D_X$-modules (as well as that of left $\D_X$-modules). 

Given a map $f:X\to Y$ of schemes, there are continuous functors
$$
\xymatrix{
f_*:\D(X)\ar[r] & \D(Y)
&
f^!:\D(Y)\ar[r] & \D(X).
}
$$
They satisfy standard composition identities: given a sequence of maps
$$
\xymatrix{
X \ar[r]^-f & Y \ar[r]^-{g} & Z
}
$$
there are canonical equivalences of functors
$(gf)_* \simeq  g_*  f_*$, $(gf)^! \simeq  f^! g^!$. (See Remark \ref{correspondences} for more
on the functoriality of the assignment $X\mapsto \D(X)$.)

Note that $\D(pt)\simeq \Vect_\CC$.
For the map $\pi:X\to pt$, we set  $\omega_X:=\pi^!\CC \in \D(X)$, so that $f^!\omega_Y \simeq \omega_X$ for any map $f:X\to Y$.
By definition, the functor of de Rham cohomology 
is given by pushforward along $\pi:X\to pt$
which we denote by
$$
\xymatrix{
\Gamma_{dR} :=\pi_*:\D(X)\ar[r] & \D(pt)\simeq \Vect_\CC
}
$$
It is representable by an object denoted by $\CC_X\in \D(X)$ in the sense that there is a functorial isomorphism
$$
\xymatrix{
\Gamma_{dR}(\fM) \simeq \Hom_{\D(X)}(\CC_X, \fM) & \fM\in \D(X)
}
$$

%

For a map $f:X\to Y$ of schemes,
the functors $f^!,f_*$  satisfy the following standard properties:


$\bullet$ Base change: given a Cartesian square
$$
\xymatrix{
X \ti_Y Z \ar[r]^-{\tilde g} \ar[d]_-{\tilde f} & X \ar[d]^-f
\\
Z \ar[r]^-g & Y
}
$$
there is a canonical equivalence of functors 
$$\tilde g_*\tilde f^! \simeq f^! g_*.
$$


$\bullet$ Projection formula:  there is a functorial equivalence
$$
\xymatrix{
f_*(f^!(\fN) \otimes \fM) \simeq f_*(\fM) \otimes \fN
& \fM\in \D(X), \fN\in \D(Y)
}$$

$\bullet$ If $f$ is proper, there is an adjunction $(f_*,f^!)$, and so in particular, $f_*$ is quasi-proper. 

$\bullet$ If $f$ is an open embedding, there is an adjunction $(f^!,f_*)$. 
More generally, if $f$ is smooth, there exists a functor 
$$
\xymatrix{
f^*:\D(Y)\ar[r] & \D(X)
}
$$ 
along with an adjunction $(f^*,f_*)$  
and canonical isomorphism 
$$
\xymatrix{
f^*\simeq f^![-2(\dim X - \dim Y)]
}
$$
In particular $f^*,f^!$ are quasi-proper, and moreover, $f^*\CC_Y\simeq\CC_X$.

$\bullet$ Kashiwara's lemma: for $i:Z\to X$ a closed embedding, $i_*$ induces 
an equivalence between $\D(Z)$ and the full subcategory $\D_Z(X) \subset \D(X)$ of  $\D$-modules on $X$ vanishing on $X\setminus Z$. The inverse is
given by the restriction of $i^!$.


\begin{remark}\label{correspondences}
As explained in \cite{FG,finiteness}, a natural setting for a complete formulation
of the structure carried by the functors $f^!,f_*$ on $\D$-modules is to consider the assignment $X\mapsto \D(X)$ as a functor
on the $\oo$-category of derived schemes with morphisms given by correspondences. This  automatically encodes 
base change (as the composition law) as well as the proper and smooth adjunctions and all the relevant compatibilities. 
\end{remark}

\subsection{Coherence, tensor product, and Verdier duality}
A  $\D$-module on a scheme $X$ is coherent if its cohomology sheaves are locally finitely generated
 $\D_X$-modules.
We
 denote by $\Dcoh(X)\subset \D(X)$ the  full subcategory of coherent $\D$-modules.

Coherent $\D$-modules are precisely the compact objects of $\D(X)$ and also generate $\D(X)$, so there
is a canonical identification 
$$\xymatrix{
\D(X)\simeq \Ind\Dcoh(X)
}
$$ Since $f_*$ for proper morphisms is quasi-proper, it preserves coherence,
and similarly since $f^*,f^!$ for smooth morphisms are quasi-proper, they preserve coherence. Moreover, the canonical objects $\omega_X, \CC_X\in \D(X)$ are coherent.

External tensor product provides a natural equivalence 
$$
\xymatrix{
\boxtimes:\D(X)\otimes \D(Y)\ar[r]^-\sim &  \D(X\times Y)
}
$$
which restricts to an equivalence 
$$
\xymatrix{
\boxtimes:\Dcoh(X)\otimes \Dcoh(Y)\ar[r]^-\sim &  \Dcoh(X\times Y)
}
$$

There is a symmetric monoidal structure on $\D(X)$ with multiplication 
$$
\xymatrix{
\otimes: \D(X)\ot \D(X) \ar[r]^-{\boxtimes} & \D(X \times X) \ar[r]^-{\Delta^!} & \D(X)}
$$
where $\Delta:X\to X\times X$ is the diagonal map.
The unit is $\omega_X\in \D(X)$.
The pullback
 $f^!$ is symmetric monoidal.

Recall from  Section~\ref{restricted op} that for a category $\cC = \Ind \cC_{cpt}$, we denote by $\cC_{cpt}^\catop$  the opposite category of $\cC_{cpt}$,
and that the category $\cC'=\Ind (\cC_{cpt}^\catop)$ is  dual to $\cC$.

Verdier duality is the unique involutive equivalence
$$\xymatrix{\DD_X:\Dcoh(X)^\catop\ar[r]^-{\sim}& \Dcoh(X)}$$
such that there is a functorial equivalence
\begin{equation}\label{Verdier}
\Hom_{\D(X)}(\DD_X(\fM),\fN)\simeq \Hom_{\D(X)}(\CC_X, \fM\ot \fN) \qquad
\fM,\fN\in \Dcoh(X)
\end{equation}
%
Note that equation~\eqref{Verdier} holds for any $\fN\in\D(X)$ by continuity.

The continuous extension of Verdier duality to all $\D$-modules provides a canonical
self-duality 
$$\xymatrix{\DD_X:\D(X)' \ar[r]^-{\sim}& \D(X) }
$$ 
Note that equation~\eqref{Verdier} now holds for any $\fM, \fN\in\D(X)$ by continuity.  In particular, equation~\eqref{Verdier} presents the counit of the canonical self-duality in the form
$$
\xymatrix{
\D(X)\otimes \D(X)\ar[r] & \D(X) & \fM\otimes \fN \ar@{|->}[r] & \Gamma_{dR}(\fM\otimes \fN)
}
$$

For a smooth morphism $f:X\to Y$, there is a canonical equivalence $f^*\simeq\DD_X f^! \DD_Y$,
and in particular, $\DD_X$ exchanges $\CC_X$ and $\omega_X$.
For a proper morphism $f:X\to Y$, there is a canonical equivalence 
$f_*\simeq \DD_Y f_* \DD_X$.\footnote{We thank N. Rozenblyum for pointing out this follows from the adjunction $(f_*,f^!)$ for proper maps together with 
the identification of $f^!$ with the transpose of $f_*$ for arbitrary maps.}


\subsection{$\D$-modules on stacks}
Throughout what follows, 
by a stack, we will always mean a quasi-compact stack  over $\CC$ with affine diagonal.
By a representable morphism, we will
always mean a quasi-compact schematic morphism.

For schemes $X$, the category $\D(X)$ satisfies   smooth descent with respect to $!$-pullbacks. 
Thus for a stack $X$, one may define the category $\D(X)$ as a limit over smooth covers. 

More explicitly
given a smooth cover $U\to X$ by a scheme, with
 induced augmented  Cech simplicial 
scheme $U_*\to X$, one obtains a cosimplicial category $\D(U_*)$, with maps given by $!$-pullbacks, such that its totalization satisfies
$$
\D(X) \simeq\lim \D(U_*)
$$ 
(See \cite[Section 7.5]{BD} for a discussion in the language of  derived categories.) 
Alternatively, one may  calculate $\D(X)$ as the limit over $*$-pullbacks (over the smooth face maps) in the above Cech
simplicial scheme. 

For an arbitrary  morphism $f:X\to Y$ of stacks, composition identities provide a continuous pullback functor $f^!$ which can be calculated locally with respect to smooth covers. In particular, there is an object $\omega_X\in \D(X)$ whose
$!$-pullback along any smooth cover $U\to X$ is the corresponding object $\omega_U\in \D(U)$.

For a representable morphism  $f:X\to Y$, base change provides a continuous functor $f_*$
which can be calculated locally with respect to smooth covers.

For a smooth representable morphism $f:X\to Y$, there is a functor $f^*$ along with an adjunction $(f^*, f_*)$
and
 canonical isomorphism 
$
f^*\simeq f^![-2(\dim X - \dim Y)].
$

One can extend the definition of $f_*$ to arbitrary  morphisms $f:X\to Y$,
but it is no longer continuous. In particular, the de Rham cohomology functor $\Gamma_{dR}$ is not continuous in general.
It is however representable by an object $\CC_X\in \D(X)$ whose $*$-pullback along any smooth cover $U\to X$ is the corresponding object $\CC_U\in \D(U)$.

%
%

\subsection{More on coherence, tensor product, and Verdier duality}

Since  $f^!$ for a smooth morphism $f:X\to Y$ preserves coherence, one may define coherent $\D$-modules
on a stack by requiring coherence locally on smooth covers. 
We
 denote by $\Dcoh(X)\subset \D(X)$ the  full subcategory of coherent $\D$-modules.

By \cite[Theorem 8.1.1]{finiteness}, for any 
stack $X$, 
 the category $\D(X)$ is compactly generated
 and all compact objects are coherent. 
However, it is 
 not necessarily true that all coherent $\D$-modules are compact. For example, the object $\CC_X \in \D(X)$ representing
the de Rham cohomology functor $\Gamma_{dR}$ is coherent (and in fact holonomic) but  not in general compact. The case of the classifying space $X=BG$ of  an affine group $G$
is illuminating and worked out in detail in \cite[Section 7.2]{finiteness}. 
In this case,  $\Gamma_{dR}$ calculates
equivariant cohomology,
which is not continuous even when $G$ is the multiplicative group. More generally, finite rank vector bundles with flat connection on a smooth stack
are coherent (and in fact holonomic) but not necessarily compact. 

In general, there is the notion of  {\em safe} stacks characterized by the assumption that 
coherent $\D$-modules are compact.
In \cite[Section 10.2]{finiteness}, the following conditions are shown to be equivalent:

$\bullet$ all coherent $\D$-modules on $X$ are compact,

$\bullet$ the object $\CC_X$ is compact, i.e., the functor $\Gamma_{dR}$ is continuous, 

$\bullet$ the identity component of the automorphism group of any geometric point of $X$ is unipotent.

For example, any Deligne-Mumford stack is safe.

More generally, there is the relative notion of safe morphism $f:X\to Y$ of stacks characterized by the assumption that $f_*$ is continuous. It guarantees  
  familiar properties 
 such as the projection formula hold.

We will  make repeated use of \cite[Corollary 8.3.4]{finiteness} which asserts that for any
stacks (in fact, for any prestacks), 
external tensor product provides a natural equivalence 
$$
\xymatrix{
\boxtimes:\D(X)\otimes \D(Y)\ar[r]^-\sim &  \D(X\times Y)
}
$$

There is  a resulting symmetric monoidal structure on $\D(X)$ whose multiplication can be calculated locally
with respect to smooth covers.

One can extend Verdier duality to $\D$-modules on stacks by applying it locally, intertwining the realizations of $\D$-modules in terms of descent under $!$- and $*$-pullbacks.
Moreover, Verdier duality restricts to compact objects, where as in equation~\eqref{Verdier}, it is characterized by a functorial equivalence
\begin{equation}\label{Verdierstacks}
\xymatrix{
\Hom_{\D(X)}(\DD_X(\fM),\fN)\simeq  \Hom_{\D(X)}(\CC_X, \fM\ot \fN) 
& \fM, \fN \in \D(X)_{cpt}
}
\end{equation}
Here as elsewhere we write $\D(X)_{cpt}$ for the compact objects of $\D(X)$.
Its continuous extension to all $\D$-modules provides a canonical self-duality 
$$\xymatrix{\DD_X:\D(X)' \ar[r]^-{\sim}& \D(X) }
$$ 
Note that equation~\eqref{Verdierstacks} does not necessarily hold for arbitrary $\fM, \fN\in\D(X)$.
The left hand side is continuous in each, but 
$\CC_X\in \D(X)$ may not be compact and hence the right hand side
may not be continuous.



\section{Integral transforms for $\D$-modules}\label{integral transforms}

We continue with our standing assumptions of preceding sections. In particular, 
by a scheme, we mean a quasi-compact, separated derived scheme.
By a stack, we mean a 
quasi-compact stack with affine diagonal, and by a representable morphism, we mean a quasi-compact schematic morphism. A stack is safe if its de Rham cohomology functor is continuous, and more generally, a morphism
is  safe if the pushforward of $\D$-modules along it is continuous.

In the previous section, we recalled that for any stacks $X_1, X_2$, external tensor product induces an equivalence
$$\xymatrix{\D(X_1)\ot \D(X_2)\ar[r]^-{\sim}& \D(X_1\times X_2).}$$
Verdier duality induces a canonical self-duality
 $$
 \xymatrix{\D(X_1)\simeq \D(X_1)'
 }$$

If in addition $X_1$ is safe, then the integral transform construction provides a description of the resulting equivalence
 $$
 \xymatrix{\D(X_1 \times X_2)\ar[r]^-\sim & \Fun^L(\D(X_1), D(X_2))
 &
 \fM \ar@{|->}[r] &  \pi_{2*}(\pi_1^!(-) \ot \fM)
 }$$

The above equivalences fit into a canonically commuting diagram
  $$
 \xymatrix{
\ar[d]_-\sim  \D(X_1)\otimes \D(X_2) \ar[r]^-\sim &  \D(X_1 \times X_2) \ar[d]^-\sim \\
  \D(X_1)'\otimes \D(X_2) \ar[r]^-\sim &   \Fun^L(\D(X_1), D(X_2))
 }$$
 Namely, for $\fM_1\in \D(X_1), \fM_2\in \D(X_2)$, there is a canonical identity
 $$
 \xymatrix{
 \pi_{2*}(\pi_1^!(-) \otimes (\pi_1^!\fM_1 \otimes \pi_2^!\fM_2)) \simeq 
 \Gamma(X_1, \fM_1 \otimes - ) \otimes \fM_2\simeq 
 \Hom_{\D(X_1)}(\DD_{X_1} \fM_1, -) \otimes \fM_2
 }
 $$
 between going around the upper right and  the lower left parts of the diagram.

 In this section, we will record some simple relative versions of the above assertions. While we will consider very restrictive
 hypotheses, the assertions often fail in any more relaxed setting.
 
 
\ssn{Functoriality over schemes}

We will not need the results of this section but they naturally fit into the theme of the paper and are likely useful elsewhere.

\begin{theorem}\label{relative product}
For a diagram $X_1\to Y \leftarrow X_2$ where $X_1,X_2$ are stacks and $Y$ a scheme,
the canonical  map induced by $!$-pullbacks is an equivalence
$$
\xymatrix{
\D(X_1)\ot_{\D(Y)} \D(X_2) \ar[r]^-\sim & \D(X_1\times_Y X_2).
}$$
\end{theorem}

\begin{proof}
The tensor product of 
$\D(Y)$-modules can be calculated
as the geometric realization of the two-sided bar
construction (\cite[4.3.4]{HA}), the simplicial category with
$k$-simplices
$$
\D(X_1)\ot \D(Y)\ot\cdots \ot \D(Y)\ot \D(X_2),
$$
where the factor $\D(Y)$ appears $k$ times
and the maps are given by the $\D(Y)$-module structures in the usual pattern. 
This is canonically equivalent to the simplicial category with $k$-simplices
$$
\D(X_1 \ti Y\ti \cdots \ti Y \ti X_2)
$$
where the factor $Y$ appears $k$ times
and the maps are given by $!$-pullbacks  in the usual pattern.

This  extends to an augmented simplicial category with augmentation given by the pullback
$$
\xymatrix{
j_{-1}^!:\D(X_1 \times X_2)\ar[r] & \D(X_1 \times_Y X_2)
}
$$
along the relative diagonal 
$$
\xymatrix{
j_{-1}:X_1\ti_Y X_2 \ar[r] & X_1 \ti X_2
}$$ 
The map $j_{-1}$ is a base change of the diagonal map of $Y$, and
since  $Y$ is a separated scheme,  $j_{-1}$ is a closed embedding. 

Similarly, the analogous diagonal maps
$$
\xymatrix{
j_k: X_1 \times Y \times \cdots \times Y \times  X_2 \ar[r] & X_1\times Y\times \cdots \times Y \times Y \times X_2
}
$$
where the factor $Y$ appears $k$ times in the domain and $k+1$ times in the codomain are closed embeddings.

Now let us consider pushforwards along the above closed embeddings and apply Kashiwara's lemma: for a closed embedding $j$
and a $\D$-module $\fM$, the natural adjunction map $\fM\to  j^!j_*\fM$ is an equivalence.
Thus by base change, 
the pushforward 
$$
\xymatrix{
j_{-1*}:\D(X_1 \times_Y   X_2) \ar[r] & \D(X_1 \ti X_2)
}
$$
and its higher analogues 
$$
\xymatrix{
j_{k *}:\D(X_1 \times Y \times \cdots \times Y \times  X_2)\ar[r] & \D(X_1\ti Y\ti \cdots \ti Y \ti Y \ti  X_2)
}$$
provide a lift to 
a split augmented simplicial category.

Finally, by~\cite[Lemma 6.1.3.16]{topos},  in any $\oo$-category, split augmented simplicial diagrams are colimit diagrams.
\end{proof}

\begin{theorem} \label{thm selfdual}
Let $X$ be a stack, $Y$ a scheme, and $f:X\to Y$ a safe morphism
(for example, a relative Deligne-Mumford stack). Then $\D(X)$ is a self-dual module over $\D(Y)$.
\end{theorem}

\begin{proof} 
 By Theorem \ref{relative product}, we have a canonical equivalence
$$
\D(X)\ot_{\D(Y)} \D(X) \simeq
\D(X\times_Y X).
$$
Thus we can define a unit and counit by the
correspondences 
$$
\xymatrix{
u=\Delta_*f^!: \D(Y)  \ar[r] & \D(X\times_Y X) \simeq \D(X)\ot_{\D(Y)} \D(X)
}
$$
$$
\xymatrix{
c=f_*\Delta^!: \D(X)\ot_{\D(Y)} \D(X)  \simeq \D(X\times_Y X) \ar[r] & \D(Y)
}
$$
where $\Delta:X\to
X\times_Y X$ is the diagonal. Here we assume $f:X\to Y$ is safe to ensure 
the continuity of the functor $f_*$.

We must check that
the following composition is the identity
$$
\xymatrix{ 
\D(X)  \ar[r]^-{u \ot \on{id}} & \D(X)\ot_{\D(Y)} \D(X)\ot_{\D(Y)} \D(X)
\ar[r]^-{\on{id}\ot c}  & \D(X). 
}
$$
Consider the following commutative diagram with Cartesian square:
$$
\xymatrix{
& 
\ar[d]_{\Delta} X \ar[rr]^-{\Delta} 
& &
X\ti_Y X \ar[r]^-{\pi_1} \ar[d]^-{\Id_1\ti \Delta_{23}} 
& X 
\\
X
& 
\ar[l]_-{\pi_2} X\ti_Y X \ar[rr]^-{ \Delta_{12}\ti \Id_3}
& &
X \ti_Y X\ti_Y X
& 
\\
}
$$

Using standard identities for composition and base change, we calculate
\begin{eqnarray*}
(\on{id}\ot c)\circ (u \ot \on{id}) & = & 
\pi_{1*} (\Id_1\ti \Delta_{23})^!(\Delta_{12}\ti \Id_3)_*\pi^!_{2}\\
& \simeq & \pi_{1*} \Delta_*\Delta^!\pi^!_{2}\\
& \simeq & \Id_{\D(X)}.
\end{eqnarray*}
\end{proof}

If $X_1\to Y$ is safe as in the previous theorem, and $X_2\to Y$ is arbitrary,   then the integral transform construction gives linear functors
 $$
 \xymatrix{\D(X_1 \times_Y X_2)\ar[r] & \Fun^L_{\D(Y)}(\D(X_1), D(X_2))
 &
 \fM \ar@{|->}[r] &  \pi_{2*}(\pi_1^!(-) \ot \fM)
 }$$
Its evident compatibility  
with the self-duality of Theorem~\ref{thm selfdual} and the identification
of Theorem~\ref{relative product} immediately implies the following.

\begin{cor}\label{cor selfdual}
Let $X_1, X_2$ be stacks, $Y$ a scheme,  $X_1\to Y$ a safe morphism, and $X_2\to Y$  an arbitrary morphism. 
Then the natural maps are equivalences
$$
\xymatrix{
\D(X_1) \ot_{\D(Y)} \D(X_2) 
\ar[r]^-\sim &  
\D(X_1\ti_Y X_2) 
\ar[r]^-\sim & 
\Fun^L_{\D(Y)}(\D(X_1), \D(X_2))
}$$
factoring the self-duality of $\D(X_1)$ over $\D(Y)$.
\end{cor}


\ssn{Functoriality over classifying stacks}\label{over stacks}

In this section, we present an equivariant generalization of 
Theorem~\ref{thm selfdual} and its corollary realizing linear functors via 
integrals kernels on the fiber product. Unfortunately, results of this nature do not hold in great generality,  but
we will be content with the case when the base stack $Y$ is the classifying stack $BG$ of an 
affine group scheme $G$. What we specifically use is the  restrictive property that
the diagonal map of such stacks is smooth.


\begin{prop}\label{BG semirigid}
For a  stack $Y$ with smooth diagonal, the monoidal category $\D(Y)$ is semi-rigid with a canonical pivotal structure.
\end{prop}

\begin{proof}
Since $\D(Y)$ is compactly generated and symmetric monoidal, 
it suffices to check that any compact object is dualizable. 

For any compact object $\fM\in \D(Y)$, we will show that the monoidal dual is the shifted Verdier dual $\DD_Y(\fM)[2\dim Y] \in \D(Y)$. 

Following the discussion of Section~\ref{duals}, it suffices to establish for any  compact object $\fM\in \D(Y)$, 
and arbitrary objects $\fL, \fN \in \D(Y)$, that there is a functorial equivalence
\begin{equation}\label{dual eq}
\xymatrix{
\Hom_{\D(Y)}(\fL \otimes \fM, \fN)
 \simeq \Hom_{\D(Y)}(\fL, \fN\otimes \DD_Y(\fM)[2\dim Y] )
}
\end{equation}

Observe that it suffices to establish the equivalence for $\fL, \fN \in \cH$ compact objects. For arbitrary objects $\fL, \fN\in \cH$, it then follows by continuity by taking colimits first in $\fN$ then in $\fL$.

Using that the diagonal $\Delta:Y\to Y\times Y$ is smooth of relative dimension $-\dim Y$, note the following functorial equivalence
\begin{eqnarray*}
\DD_Y(\fL \otimes \fM) & = & \DD_Y\Delta^!(\fL \boxtimes \fM)\\
&  \simeq & \Delta^*(\DD_Y(\fL) \boxtimes \DD_Y(\fM))\\
& \simeq &  \Delta^!(\DD_Y(\fL) \boxtimes \DD_Y(\fM))[2\dim Y] \\
& \simeq &  \DD_Y(\fL) \otimes \DD_Y(\fM)[2\dim Y]
\end{eqnarray*}

Now returning to the sought after equivalence~\eqref{dual eq}, we find
functorial equivalences
\begin{eqnarray*}
\Hom_{\D(Y)}(\fL \otimes \fM, \fN) &\simeq &\Hom_{\D(Y)}(\CC_Y, \DD_Y(\fL \otimes \fM) \otimes \fN) \\ 
&\simeq&  \Hom_{\D(Y)}(\CC_Y, (\DD_Y(\fL) \otimes \DD_Y(\fM) [2\dim Y]) \otimes \fN)\\ 
&\simeq&\Hom_{\D(Y)}(\fL,  \DD_Y(\fM)[2\dim Y]\otimes  \fN) 
\end{eqnarray*}
\end{proof}

\begin{remark}
Implicit in the  proposition is the Poincar\'e duality  identity $\DD_Y(\omega_Y)[2\dim Y] \simeq \omega_Y$ since the dual of the monoidal unit must be the monoidal unit. \end{remark}

\begin{cor}\label{selfdual over stack}
Let $Y$ be a  stack with smooth diagonal, $X$ a stack, and  $X\to Y$ an arbitrary morphism.
Then $\D(X)$ is
self-dual as a $\D(Y)$-module. 
\end{cor}

\begin{proof}
As a plain category, $\D(X)$ is self-dual, so the previous proposition and Proposition~\ref{bimodule dual} imply the corollary.
\end{proof}

If $X_1, X_2$ are stacks, $X_1\to Y $ is safe and $X_2\to Y$ is arbitrary, then 
there is a canonical map induced by $!$-pullbacks
$$
 \xymatrix{
 \Psi: \D(X_1)\otimes_{\D(Y)} D(X_2) \ar[r] &  \D(X_1 \times_Y X_2)
  }$$
and the integral transform construction gives linear functors
 $$
 \xymatrix{
 \Phi:\D(X_1 \times_Y X_2)\ar[r] & \Fun^L_{\D(Y)}(\D(X_1), D(X_2))
 &
 \Phi_\fM = 
 \pi_{2*}(\pi_1^!(-) \ot \fM)
 }$$
Tracing  back through the constructions and applying standard identities, the functors $\Psi,\Phi$ evidently provide a factorization
$$\xymatrix{&\D(X_1 \times_Y X_2)\ar[dr]^-{\Phi}&\\
\D(X_1)\ot_{\D(Y)} \D(X_2) \ar[rr]^-{\sim} \ar[ur]^-{\Psi}  &&\Fun^L_{\D(Y)}(\D(X_1), \D(X_2))}
$$
of the equivalence provided by Corollary~\ref{selfdual over stack}. In general neither $\Psi,\Phi$  is itself an equivalence, but we have the following.

\begin{prop}\label{prop int trans adjoint}
Let $Y$ be a smooth stack with smooth diagonal, $X_1, X_2$ stacks, $X_1\to Y$ a safe morphism, and $X_2\to Y$  an arbitrary morphism. 

Then the integral transform $\Phi$ admits a fully faithful left adjoint $\Phi^L$.
\end{prop}

\begin{proof}
Tracing back through the algebraic arguments leading to Proposition~\ref{bimodule dual}, 
one sees that the  functor category $\Fun^L_{\D(Y)}(\D(X_1), \D(X_2))$ can be calculated as the limit of the cosimplicial category
with $k$-cosimplices 
$$
\D(X_1 \ti Y\ti \cdots \ti Y \ti X_2)
$$
where the factor $Y$ appears $k$ times
and the maps are given by pushforwards  in the usual pattern.
This naturally extends to an augmented cosimplicial category with augmentation given by the pushforward
$$
\xymatrix{
j_{-1*}: \D(X_1 \times_Y X_2) \ar[r] & \D(X_1 \times X_2)
}
$$
along the relative diagonal 
$$
\xymatrix{
j_{-1}:X_1\ti_Y X_2 \ar[r] & X_1 \ti X_2
}$$ 

By base change, this augmented cosimplicial category is left adjointable in the sense of
\cite[Corollary 6.2.4.3]{HA}, thus  the integral transform construction
$$
\xymatrix{
\Phi:\D(X_1\times_Y X_2) \ar[r] & \Fun_{\D(Y)}(\D(X_1), \D(X_2))
}$$
admits a  fully faithful left adjoint.

\end{proof}






\section{Hecke categories and character sheaves}\label{character section}

In this section, we study the homotopical algebra of Hecke categories
arising in representation theory. 
Our focus is the  Hecke
category  of $\D$-modules on the double quotient stack $B\bs G/B$
of a reductive group $G$ by a Borel subgroup $B\subset G$,
and its unipotent monodromic variant.
Our aim is to construct a topological field theory from these Hecke categories, 
and in particular to relate their traces and centers with character sheaves.

If we introduce the classifying stacks $X=BB$, $Y=BG$,
then the double quotient stack can be realized as the fiber product
$$B\bs G/B\simeq \XYX$$  
From this point of view, the
special technical features of this setup are the following:
\begin{enumerate}
\item The morphism $p:X\to Y$ is proper.
\item The diagonal $\delta:X\to X\ti X$ is smooth.
\end{enumerate}
The second property is especially restrictive though satisfied by
classifying spaces of smooth group schemes.
The above conditions will imply that we have a sufficient arsenal of duality and
adjunctions.

\subsection{Hecke categories}\label{convolution algebras}
We establish here that 
Hecke categories 
are semi-rigid categories with a canonical Calabi-Yau structure, and thus two-dualizable monoidal categories defining oriented
topological field theories.


To begin, let $p:X\to Y$ be a morphism of stacks. Introduce the Hecke category $ \cH=\D(X\times_Y X)$. 
Consider the
convolution diagram
$$
\label{convolution diagram}
\xymatrix{
& \ar[dl]_-{p_{12}}\ar[dr]^-{p_{23}} \ar[d]_{p_{13}}X \times_Y X \times_Y X & \\
\XYX & \XYX & \XYX }
$$
Equip $\cH=\D(X\times_Y X)$ with the monoidal product defined by
convolution
$$
\xymatrix{
\star:\cH \otimes \cH \ar[r] & \cH
}$$
$$
\xymatrix{
\fM\star \fN = p_{13*}(p_{12}^!(\fM) \ot p_{23}^!(\fN))
\simeq p_{13*}(p_{12} \ti p_{23})^!(\fM \boxtimes \fN)
}$$


Similar diagrams provide the usual associativity compatibilities of an algebra object.
The monoidal unit $1_\cH$  is given by the pushforward $u_*\omega_X$ along the map
$u:X\to \XYX$ induced by the diagonal $\delta:X \to X\ti X$.

Recall that if $p:X\to Y$ is safe, then the integral transform construction gives a monoidal  functor
$$
\xymatrix{
\cH\ar[r] &  \Fun_{\D(Y)}(\D(X), \D(X))
&
\fM \ar@{|->}[r] & p_{2*}(p_1^!(-) \ot \fM)
}$$

\begin{remark} 
More formally, the monoidal structure on $\D(X\times_Y X)$ can be defined as follows. 
The stack $\XYX$ can be identified with endomorphisms of $X$ in the $\infty$-category
of stacks over $Y$ with morphisms given by correspondences. The assignment
of categories of $\D$-modules extends to a symmetric monoidal functor on this $\oo$-category (see Remark~\ref{correspondences}),
hence induces an associative algebra structure on $\D(\XYX)$.
\end{remark}

It is important to note that all of the maps involved in the monoidal structure of $\cH$ are obtained by base change from either 
the diagonal map $\delta:X\to X \ti X$ or the projection $p:X\to Y$. Moreover, the monoidal structure only uses pullbacks along maps
of the first type and pushforwards along maps of the second type. Thus if we assume $\delta$  is smooth
and $p$ is proper,  the functors $\delta^!$ and $p_*$, and hence the monoidal product on $\cH$, will be quasi-proper. 
In addition, $\D(X)$ is semi-rigid under tensor product,
by Proposition~\ref{BG semirigid}.

\begin{theorem}\label{perfect Hecke}
Suppose $X$ is a  stack with smooth diagonal, $Y$ is a stack, and $X\to Y$ is a proper representable map.

Then the Hecke category  $\cH=\D(\XYX)$
 is semi-rigid (hence 2-dualizable).

\end{theorem}

\begin{proof}
To prove the theorem, let us first establish a candidate for the dual of a compact object.

First, recall that Verdier duality  $\DD_{\XYX}$
provides a self-duality
$$
\xymatrix{
\cH=\D(\XYX) \simeq \D(\XYX)' = \cH'
}
$$
By definition, the compact objects of $\cH$ and $\cH'$ form opposite categories to each other.

Next consider the swap involution
$$
\xymatrix{
\sigma: \XYX \ar[r]^-\sim & \XYX
& \sigma (x_1, y_{12}, x_2) = (x_2, y_{12}^{-1}, x_1)
}
$$
where $y^{-1}_{12}$ is the path $y_{12}$ traced in the opposite direction.
The induced equivalence on $\D$-modules
intertwines the algebra structure and the opposite algebra structure
$$
\xymatrix{
\sigma:\cH \ar[r]^-\sim & \cH^{op}.
}
$$

Finally, introduce the composite of the swap involution, Verdier duality and shift 
$$
\xymatrix{
\iota : \cH\ar[r]^-\sim & \cH'
&
\iota(\fM) = \DD_{\XYX}(\sigma(\fM)) [2\dim X]
}
$$

We will show that any compact object $\fM\in \cH$ is left and right
dualizable, with both duals identified with $\iota(\fM) \in \cH$,
thereby establishing that $\cH$ is semi-rigid.
Following the discussion of Section~\ref{duals}, it suffices to establish for a  compact object $\fM\in \cH$, 
and arbitrary objects $\fL, \fN \in \cH$, that there are functorial equivalences
$$
\xymatrix{
\Hom_{\cH}(\fL \star \fM, \fN)
 \simeq \Hom_{\cH}(\fL, \fN\star \iota(\fM))
& 
\Hom_{\cH}(\fM\star \fL , \fN)
\simeq   \Hom_{\cH}(\fL, \iota(\fM) \star\fN )
}
$$

We will give a proof of the first. The second is similar or follows by applying $\sigma$ to the first and renaming objects.

Observe that it suffices to establish the equivalence for $\fL, \fN \in \cH$ compact objects. For arbitrary objects $\fL, \fN\in \cH$, it then follows by continuity by taking colimits first in $\fN$ then in $\fL$.

For $\fL, \fM, \fN \in \cH$ compact objects, we will establish the equivalent identity
$$
\xymatrix{
\Hom_{\cH}(\CC_{\XYX}, \DD_{\XYX}(\fL \star \fM)\otimes \fN)
 \simeq \Hom_{\cH}(\CC_{\XYX}, \DD_{\XYX}(\fL)\otimes (\fN\star \iota(\fM)))
}
$$

Since $p_{13}$ is a base change of $p$, it is  proper and hence $p_{13*}$ commutes with Verdier duality.
Since $p_{12}\times p_{23}$ is a base change of $\delta$, it is smooth of relative dimension $-\dim X$
and hence $(p_{12} \times p_{13})^* \simeq (p_{12} \times p_{13})^![2\dim X]$.
Hence we have the identity
\begin{eqnarray*}
 \DD_{\XYX}(\fL \star \fM)
&\simeq  &  \DD_{\XYX}(p_{13*}((p_{12} \times p_{13})^!(\fL \boxtimes  \fM)))\\
&\simeq  &  p_{13*}\DD_{\XYX}(p_{12} \times p_{13})^!(\fL \boxtimes  \fM)\\
&\simeq  &  p_{13*}(p_{12} \times p_{13})^*(\DD_{\XYX}(\fL) \boxtimes  \DD_{\XYX}(\fM))\\
&\simeq  &  p_{13*}(p_{12} \times p_{13})^!(\DD_{\XYX}(\fL) \boxtimes  \DD_{\XYX}(\fM)) [2\dim X]\\
&\simeq  &  \DD_{\XYX}(\fL) \star  \DD_{\XYX}(\fM) [2\dim X]\\
\end{eqnarray*}

Substituting this into the above sought after identity and renaming objects, 
we seek to establish the following general identity
\begin{equation}\label{general eq}
\Gamma_{\XYX}((\fL \star \fM)\otimes \fN)
  \simeq  \Gamma_{\XYX}(\fL\otimes (\fN\star \sigma(\fM)))
\end{equation}

Let us first modestly reformulate how to calculate each side, in particular the quantities obtained before taking global sections. For the left hand side, consider the commutative diagram whose right hand square is Cartesian
$$
\xymatrix{
 && 
\ar[d]^-{\Id \times p_{13}} \ar[dll]_-{\hspace{-3em} p_{12} \times p_{23} \times p_{13}}  (X\times_Y X\times_Y X)  \ar[r]^-{p_{13}} & \XYX\ar[d]^-\Delta\\
(\XYX)^3 && \ar[ll]^-{p_{12} \times p_{23} \times p_{45}} (X\times_Y X\times_Y X) \times (\XYX) \ar[r]_-{p_{13} \times p_{45}} & (\XYX)^2 
}
$$
By a composition identity in the triangle and base change in the right hand square, there is a functorial equivalence 
$$
(\fL \star \fM)\otimes \fN \simeq p_{13*}(p_{12}\times p_{23} \times  p_{13})^!
(\fL\boxtimes \fM\boxtimes \fN)
$$

Similarly, for the right hand side, consider the commutative diagram whose right hand square is Cartesian
$$
\xymatrix{
&& 
\ar[d]^-{p_{13} \times \Id} \ar[dll]_-{\hspace{-3em} p_{13} \times \sigma p_{23} \times  p_{12}}  (X\times_Y X\times_Y X)  \ar[r]^-{p_{13}} & \XYX\ar[d]^-\Delta\\
(\XYX)^3 && \ar[ll]^-{p_{12} \times \sigma p_{45} \times  p_{34}} (\XYX) \times  (X\times_Y X\times_Y X)\ar[r]_-{p_{12} \times p_{35}} & (\XYX)^2 
}
$$
By a composition identity in the triangle and base change in the right hand square, there is a functorial equivalence 
$$
\fL \otimes ( \fN\star \sigma(\fM)) \simeq p_{13*} (p_{13} \times \sigma p_{23} \times  p_{12})^!
(\fL\boxtimes \fM\boxtimes \fN)
$$

Finally, observe that there is a natural commutative diagram 
$$
\xymatrix{
\ar[drr]^-\sim_{r} \ar[d]_-{p_{13} \times \sigma p_{23} \times  p_{12}}  (X\times_Y X\times_Y X)  \ar[rr]^-{p_{13}} &&  \XYX\\
(\XYX)^3 && \ar[ll]^-{p_{12} \times  p_{23} \times  p_{13}}  (X\times_Y X\times_Y X)  \ar[u]_-{p_{13}} 
}
$$
where the diagonal map $r$ is the rotational equivalence given by
$$
\xymatrix{
r (x_1, y_{12}, x_2, y_{23}, x_3) = (x_1, y_{12} \circ y_{23}, x_3, y^{-1}_{23}, x_2)
}
$$
where the path $y_{12} \circ y_{23}$ is the composition of the paths $y_{12}, y_{23}$, and the path
$y^{-1}_{23}$ is the path $y_{23}$ traced in the opposite direction.

Let us write $\pi:\XYX\to pt$.
We have seen that the two sides of the sought after equivalence~\eqref{general eq} can be realized by going around the edges of the above diagram
$$
\Gamma_{\XYX}((\fL \star \fM)\otimes \fN) \simeq 
\pi_*p_{13*}(p_{12} \times p_{23} \times p_{13})^!(\fL\boxtimes \fM\boxtimes \fN)
$$
$$
 \Gamma_{\XYX}(\fL\otimes (\fN\star \sigma(\fM)))
 \simeq \pi_*p_{13*}(p_{13} \times \sigma p_{23} \times p_{13})^!(\fL\boxtimes \fM\boxtimes \fN)
$$
Thus the rotational equivalence $r$ intertwines them and we obtain the  equivalence~\eqref{general eq}.
\end{proof}

\begin{theorem}\label{CY Hecke}
Suppose $X,Y$ are as in Theorem~\ref{perfect Hecke}.
Then the Hecke category  $\cH=\D(\XYX)$
carries a canonical Calabi-Yau structure.

\end{theorem}

\begin{proof}
Let us briefly recall from~\cite[Section 5.3.3]{HA} the realization of Hochschild homology and its $S^1$ symmetry from the point of view of topological chiral homology.

First, given a framed circle $S$, for example, the standard circle $S^1=[0,1]/\hspace{-0.25em}\sim$ with its natural induced orientation, let $I_S$ be the $\infty$-category given by the nerve of the poset of framed embeddings of finite disjoint unions of open intervals in $S$ with partial order given by framed inclusions.
Given the $E_1$-algebra $\cH$, introduce the natural functor 
$$
\xymatrix{
\psi_\cH:I_S\ar[r] &  \St_\CC
& \psi_{\cH}(I) = \cH^{\otimes \pi_0(I)}
}$$
whose structure on morphisms is given by the $E_1$-structure on $\cH$.
By definition, the topological chiral homology of $\cH$ over $S$ is the colimit
$$
\xymatrix{
\int_{S} \cH = \colim_{I_S} \psi_{\cH}
}$$
Observe that  the above makes sense for families of framed circles, and so exhibits $\int_{S} \cH$ as the fiber of an $\oo$-local system  over
the moduli of framed circles $B\mathit{Diff^+}(S^1)$.

Recall as well that for the standard circle $S^1=[0,1]/\hspace{-0.25em}\sim$, the realization of the bar resolution in terms of subdivisions of the interval $[0,1]$ provides a natural 
morphism which is an equivalence
$$
\xymatrix{
\Tr(\cH) \simeq | \cH\otimes_{\cH \otimes \cH^{op}} C_*(\cH)| \ar[r]^-\sim & \colim_{I_{S^1}} \psi_{\cH} = \int_{S^1} \cH
}
$$

Now to endow $\cH$ with a Calabi-Yau structure,  let $\SS\to B\mathit{Diff^+}(S^1)$ denote the universal framed circle, and let $\mathbb{V}\hspace{-0.25em}\operatorname{ect}_\CC$ denote the constant $\oo$-local system on 
$B\mathit{Diff^+}(S^1)$ with fiber $ \Vect_\CC$. We must 
 construct a morphism of $\oo$-local systems
$$
\xymatrix{
\tau:\int_{\SS} \cH \ar[r] &\mathbb{V}\hspace{-0.25em}\operatorname{ect}_\CC
}
$$
Moreover, for a given framed circle $S$ and pair of disjoint open intervals $I_1 \coprod I_2 \subset S$  (hence any framed circle with any  pair of disjoint open intervals), the natural composite map 
$$
\xymatrix{
\cH \otimes \cH \simeq \cH^{\pi_0(I_1\coprod I_2)} \ar[r] & \int_{S} \cH \ar[r]^-{\tau_S} &\mathbb{V}\hspace{-0.25em}\operatorname{ect}_\CC
}
$$
must be the evaluation pairing of a self-duality of $\cH\in \St_\CC$.

To achieve this, 
given an object $I = \coprod_i I_i \in I_S$, let $\cM(I)$ be the stack of maps $\coprod_i  I_i\to Y$ with lifts $\coprod_i \partial I_i \to  X$ at boundary points. (Note we do not regard the boundary points of the intervals as embedded in $S$ since they may coincide as, for example, for an interval of the form $I = S\setminus\{s\}$ where $s\in S$ is a single point.)  Let $D$ be a closed disk with $\partial D = S$, and
 $\cM(D, I)$ be the stack of maps $D\to Y$ with a lift over the complement $S\setminus (\coprod_i  I_i) \to  X$.
 Consider the natural correspondence
 $$
 \xymatrix{
 \cM(I) & \ar[l]_-p \cM(D, I) \ar[r]^-q & pt
 }
 $$
  Observe that $\psi_\cH(I) = \D(\cM(I))$ and thanks to standard identities for $\D$-modules, the functor
  $$
  \xymatrix{
  \D(\cM(I))\ar[r] &  \Vect_\CC &  M \ar@{|->}[r] & q_*p^!M 
  }
  $$
  provides a natural augmentation to the diagram given by $\psi_\cH$ and hence a functor
  $$
  \xymatrix{
\tau_S:\int_{S} \cH \ar[r] &\Vect_\CC
}
  $$
Since  the above makes sense for families of framed circles, 
it descends to the sought after
morphism of $\oo$-local systems
$$
\xymatrix{
\tau:\int_{\SS} \cH \ar[r] &\mathbb{V}\hspace{-0.25em}\operatorname{ect}_\CC
}
$$

Finally, consider the natural composite map
$$
\xymatrix{
\cH \otimes \cH \simeq \cH^{\pi_0(I_1\coprod I_2)} \ar[r] & \int_{S} \cH \ar[r]^-{\tau_S} &\mathbb{V}\hspace{-0.25em}\operatorname{ect}_\CC
}
$$
Unwinding the above, we find it is given by 
  the functor
  $$
  \xymatrix{
 \cH \otimes \cH \simeq \D(\XYX \times \XYX) \ar[r] &  \Vect_\CC
 &
  M \ar@{|->}[r] & q_*p^!M 
  }
  $$
where $p:\XYX \to \XYX \times \XYX$ is the diagonal and $q:\XYX\to pt$ is the projection.
Observe that under the identification of $\cH\in \St_\CC$ with its dual given by Verdier duality, the above morphism
is precisely the evaluation pairing.
\end{proof}

\begin{remark}
We will soon identify $\Tr(\cH) \simeq \int_{S} \cH$ with a full subcategory of $\D$-modules on the loop space $\cL Y$.
Consider the natural correspondence of constant loops
$$
\xymatrix{
\cL Y & \ar[l]_-e Y \ar[r]^-\pi & pt
}
$$
Thanks to standard identities for $\D$-modules, 
the trace $\tau_S:\int_{S} \cH\to \Vect_\CC$ will factor through the functor
  $$
  \xymatrix{
  \D(\cL Y)\ar[r] &   \Vect_\CC & M \ar@{|->}[r] & \pi_*e^!M
  }
  $$
\end{remark}

The above theorems thus allow us to apply the
Cobordism Hypothesis \cite{jacob TFT} to the Hecke category to obtain
an oriented two-dimensional topological field theory.

\begin{cor}\label{Hecke TFT}
 There is a unique symmetric monoidal functor
$\cZ: 2Bord \to Alg_{(1)}(\St_\CC)$ 
with $\cZ(pt)=\cH$. 
\end{cor}



\ssn{Loop spaces and horocycle correspondences}

\newcommand{\perpK}{{}^\perp{\cK}}

We establish here the basic relation between Hecke categories
and loop spaces. It is an abstraction of the horocycle correspondence
from representation theory. 
We will continue with the assumptions that $p:X\to Y$ is proper and
$X$ has smooth diagonal $\delta:X\to X\times X$. 

Consider the fundamental correspondence
$$
\xymatrix{
 \cL Y  = Y\ti_{Y\ti Y} Y & \cL Y \ti_Y X =X \ti_{X \ti Y} X \ar[l]_-{q} \ar[r]^-{\epsilon} & X\ti_Y X.\\
}
$$
where $q$ is a base change of $p$, and $\epsilon$ is a base change $\delta$.
%
%

Define the functor
$$
\xymatrix{
F: \D(X\ti_Y X) \ar[r]^-{\epsilon^!} & \D(\cL Y \ti_Y X) 
 \ar[r]^-{q_*}& \D(\cL Y) \\
}
$$
Since $q$  is a base change of $p$ and hence proper, there is an adjunction $(q_*,q^!)$, and since $\epsilon$ is a base change of $\delta$ and hence
smooth, of relative dimension $-\dim X$, there is an equivalence $\epsilon^*\simeq\epsilon^![2\dim X]$
and
 an adjunction $(\epsilon^*,\epsilon_*)$.
Thus there is the continuous right adjoint
$$
\xymatrix{
F^r:\D(\cL Y)   \ar[r]^-{q^!}&\D(\cL Y \ti_Y X)  \ar[rr]^-{\epsilon_*[-2\dim X]} && \D(X\ti_Y X)\\
}
$$ 


Let $\cI(F) \subset \D(\cL Y)$ denote the cocompletion of the  image of $F$.
It is naturally equivalent to the left orthogonal ${}^\perp \cK(F^r) \subset \D(\cL Y)$ to the kernel of $F^r$.
Since $q^!$ satisfies descent, both in turn are naturally equivalent to the full subcategory of $ \D(\cL Y)$ of objects whose image under $q^!$  is in the cocompletion of the image of $\epsilon^!$.

\begin{theorem}\label{thm center}
There is a natural commutative diagram
$$
\xymatrix{
 \D(\XYX) & \\
{}^\perp \cK(F^r) \ar[u]^-{F^r} \ar[r]^-\sim&  \ar[ul]_-\fz \Z(\D(\XYX)) 
}
$$
\end{theorem}

\begin{proof}
For notational convenience, set $\cA = \D(X \times_Y X)$ and $\cB = \D(X)$.  
Observe that pushforward along the relative diagonal $ \cB \to \cA$ is monoidal, and thus we may regard $\cA$ as an algebra in $\cB$-bimodules. 
Given an algebra $\cA$ in $\cB$-bimodules, we have its relative bar resolution
 $$
  \xymatrix{
  \cA \simeq 
  \left| \cA^{\otimes_\cB (\bullet+2)} \right| 
  }
  $$
   which can be used to calculate its center
  $$  \xymatrix{
\Z(\cA) =  \Hom_{\cA \otimes \cA^{op}}(\cA, \cA) = \Hom_{\cA \otimes \cA^{op}}(\left| \cA^{\otimes_\cB (\bullet+2)} \right|,  \cA) = 
  \Tot\left\{\Hom_{\cB \otimes \cB^{op}} (\cA^{\otimes_\cB \bullet},  \cA) \right\}
}  
$$
We will access the center as the totalization of the cosimplicial object 
$$
\cC^\bullet =  \Hom_{\cB \otimes \cB^{op}} (\cA^{\otimes_\cB \bullet},  \cA)
$$

Unwinding the notation and using the canonical identity $\D (X)^{\otimes k} \simeq \D(X^k)$,  we find the cosimplicial category
\[ 
\xymatrix{
\cC^\bullet \simeq 
\Fun_{\D(X\times X)}^L(\D( (X \times_Y X)^{\times_X\bullet } ), \D(X \times_Y X))
}
\]

 Using the canonical identification of functors with integral transforms of Proposition~\ref{prop int trans adjoint}, we have a level-wise fully faithful map
 of cosimplicial diagrams
 \[ 
\xymatrix{
\cC^\bullet \ar[r] &  
 \D(X^{\times_Y (\bullet+1)} \times_Y \cL Y)
 }
\]
where  the cosimplicial  structure maps on the latter are given by $!$-pullback functors.

Thus the center $\Z(\cA)$ is equivalent to the full subcategory
 of $\D(\cL Y)$ of objects whose image under $q^!$  is in the cocompletion of the image of $\epsilon^!$.
\end{proof}

\begin{cor}\label{cor trace}
There is a natural commutative diagram
$$
\xymatrix{
& \ar[dl]_-\ftr \D(\XYX) \ar[d]^-{F} &\\
\Tr(\D(\XYX)) \ar[r]^-\sim   & \cI(F)
}
$$
\end{cor}

\begin{proof}
We have seen that $ \D(X\ti_Y X)$ is semi-rigid and pivotal, and hence its center and trace are equivalent.
More precisely, the augmented semi-cosimplicial diagram calculating the center can be obtained from the augmented semi-simplicial diagram calculating the trace by passing to right adjoints. In particular, the universal central functor can be obtained as the right adjoint of
the universal trace. Thus passing to left adjoints in the diagram of the previous theorem provides
the diagram of the corollary.
\end{proof}


\ssn{Character sheaves}

Now we  apply our previous results
to our motivating example when $X=BB$, $Y=BG$, for a reductive group $G$  and   Borel subgroup $B\subset G$. 
Thus we are studying the Hecke category $\cH_G = \D(B\bs G/B)$
of Borel bi-equivariant $\D$-modules on  $G$.

With this setup, the fundamental correspondence is the horocycle correspondence
$$
\xymatrix{
G/G & \ar[l]_-{q} (G\ti G/B)/G \ar[r]^-{\epsilon} & B\bs G/B.
}
$$ 
It contains the traditional Springer correspondence
as a subspace
$$
\xymatrix{
G/G & \ar[l]_-{} \wt G/G\simeq B/B \ar[r]^-{} & B\bs B/B \simeq BB
}
$$ 
where $\wt G \subset G \ti G/B$ comprises pairs where the group element fixes the flag.

%

\begin{definition}
The dg category $\chsh_G$ of unipotent character sheaves is
defined to be the full subcategory of $\D(G\adjquot G)$ generated
(under colimits) by the image of the Harish
Chandra transform~$F$.
\end{definition}

 Applying  our previous results, we immediately obtain the following. 

\begin{theorem}\label{thm char shvs}
The category of unipotent character sheaves $\chsh_{G}$
is equivalent to both the trace and center 
of the Hecke category $\cH_{G}$.
The equivalences fit into natural adjoint commutative diagrams
$$
\xymatrix{
&  \cH_{G} \ar[dl]_-\ftr \ar[d]<-1ex>_-{F}&\\
 \Tr(\cH_{G}) \ar[r]^-\sim & \ar[u]<-1ex>_-{F^r} \chsh_{G} \ar[r]^-\sim
 &  \ar[ul]_-\fz \Z(\cH_G)
   }
$$
 \end{theorem}


\section{Monodromic variant}\label{monodromic section}
In this section, we extend our previous results to Hecke categories of monodromic $\D$-modules.
For simplicity, we will  only consider unipotent monodromy, though  generalizations to arbitrary monodromies should hold. Most of what we state is purely algebraic and could be presented in far greater generality.


\subsection{Monodromic $\D$-modules}

Let $X$ be a stack and $\cA_X\in \D(X)$ a commutative algebra object.

We refer to objects of the category of $\cA_X$-modules 
$$\D_{\cA_X}(X) := \Mod_{\cA_X}(\D(X))
$$ as monodromic
$\D$-modules.

A smooth map $\pi:\tilde X \to X$ provides a natural source of such algebras.
The Barr-Beck-Lurie theorem applied to the adjunction 
$$\xymatrix{ \pi^*: \D(X) \ar[r]<+.5ex> & \ar[l]<+.5ex>  \D(\Xtil) : \pi_*}$$
provides an equivalence
of the full subcategory of $\D(\tilde X)$ generated by $*$-pullbacks from $\D(X)$  with modules over the monad $T=\pi_*\pi^*$
acting on $\D(X)$. 
By the projection formula,
the monad $T$ is  represented by the algebra object
$$
\cA_X = \pi_*\pi^*\omega_X  \in \D(X)
$$
Thus one can view  $\D_{\cA_X}(X)$ as the $\D$-module affinization of $\tilde X$.

\begin{example}
If $\pi:\tilde X\to X$ is a torsor for a torus $H$, then a monodromic $\D$-module on $X$ 
in the above sense is equivalent to a traditional $H$-monodromic $\D$-module on $\tilde X$ with unipotent monodromy around $H$.
\end{example}

Recall by Proposition~\ref{BG semirigid} that if $X$ has smooth diagonal, then the monoidal category $\D(X)$ is semi-rigid with a 
canonical pivotal structure. Since the algebra $\cA_X$ is commutative, the category $\D_{\cA_X}(X)$ is naturally 
symmetric monoidal with unit $\cA_X$. We have the following algebraic observation extending Proposition~\ref{BG semirigid}.

\begin{prop}
For a stack $X$ with smooth diagonal, the 
 monoidal category $\D_{\cA_X}(X)$ is semi-rigid with a 
canonical pivotal structure.
\end{prop}

\begin{proof}
For compact objects $\fM\in \D(X)$, objects of the form $\cA_X \otimes \fM\in \D_{\cA_X}(X)$ are compact and generate. Recall that $\fM^\vee = \DD_X(\fM)[2\dim X]\in \D(X)$ is the monoidal dual of $\fM \in \D(X)$. It is simple to check that 
$\cA_X\otimes \fM^\vee\in \D_{\cA_X}(X)$ is the monoidal dual of $\cA_X \otimes \fM \in \D_{\cA_X}(X)$.
\end{proof}


\subsection{Monodromic Hecke categories}

Let us return to the setup of a proper map $p:X\to Y$ of stacks where
$X$ has smooth diagonal $\delta:X\to X\times X$. 
Suppose given 
$\cA_X\in \D(X)$ a commutative  algebra object. 

Let $j:\XYX \to X\times X$ be the natural map, introduce the commutative algebra object
$$
\cA^{(2)}_X = j^! (\cA_X \boxtimes \cA_X) \in \D(\XYX)
$$
and the resulting Hecke category of bimonodromic $\D$-modules
$$\D_{\cA^{(2)}_X}(\XYX) = \Mod_{\cA^{(2)}_X}(\D(\XYX))
$$
More generally, we have similar constructions for repeated fiber products.

There is a natural monoidal structure on the  bimonodromic Hecke category 
$\tilde\cH = \D_{\cA^{(2)}_X}(\XYX)$ with product given by
$$
\fM \star_{\cA_X} \fN = \tilde p_{13*}(\tilde p_{12}^!(\fM) \ot_{\cA^{(3)}_X} \tilde p_{23}^!(\fN))
$$ 
Here the functors are the ringed-space variants of the usual pullback, tensor and pushforward functors. Namely,
the $!$-pullbacks denote the composition of the usual $!$-pullbacks and then tensoring up with the free $ \cA_X$-module on the remaining factor. The tensor product 
denotes the tensor product of $\cA^{(3)}_X$-modules.
 The $*$-pushforward denotes 
 the composition of the usual $*$-pushforward and then forgetting the $\cA_X$-module structure on the
 middle factor.
The monoidal unit $1_{\tilde \cH}$ is the diagonal bimodule $u_*\cA_{X}$ arising from the map $u:X\to \XYX$. 
Standard identities produce the higher constraints of the monoidal structure.

In the case where $\cA_X = \pi_*\pi^*\omega_X  \in \D(X)$ for a smooth map $\pi:\tilde X\to X$ as discussed above,  the 
monodromic Hecke category
is a full subcategory 
$$
\xymatrix{
\tilde\cH =\D_{\cA^{(2)}_X}(\XYX) 
\ar@{^(->}[r] &
\D(\tilde X\times_Y \tilde X)
}$$
Moreover, the latter category is naturally monoidal with respect to convolution and the above inclusion is monoidal.

\begin{example}
If $\tilde X = B \to X = BB \to Y = BG$ are the natural projections of classifying stacks for a Borel subgroup $B\subset G$ and its unipotent radical $N\subset B$, 
then a bimonodromic $\D$-module  in the above sense is equivalent to a traditional unipotent bimonodromic $\D$-module on $N\bs G/N$. 
\end{example}

We have the following algebraic observation extending Theorems~\ref{perfect Hecke}, \ref{CY Hecke}.

\begin{theorem}\label{monodromic Hecke}
Suppose $X$ is a  stack with smooth diagonal, $Y$ is a stack, and $X\to Y$ is a proper representable map.
Suppose  
$\cA_X\in \D(X)$ is a commutative  algebra object whose underlying object is compact.

Then the monodromic Hecke category  $\tilde \cH=\D_{\cA^{(2)}_X}(\XYX)$
 is semi-rigid (hence 2-dualizable) with a canonical Calabi-Yau structure. 
 
\end{theorem}

\begin{proof}
For compact objects $\fM\in \cH$, objects of the form $\cA^{(2)}_X\otimes \fM\in \tilde \cH$ are compact and generate. Recall from the proof of  Theorem~\ref{perfect Hecke} that $\iota(\fM)\in \cH$ is the  left and right monoidal dual of such compact objects $\fM \in \cH$.

Recall the natural map $j:\XYX \to X\times X$, and introduce the $\cA^{(2)}_X$-modules
$$
\xymatrix{
\cA^{(2)}_X{}' = j^! (\cA_X \boxtimes \cA^\vee_X) \in \D(\XYX)
&
{}'\hspace{-0.25em}\cA^{(2)}_X = j^! (\cA^\vee_X \boxtimes \cA_X) \in \D(\XYX)
}$$

It is simple to check that 
$\cA^{(2)}_X{}' \otimes \iota(\fM)\in \tilde \cH$
(respectively, ${}'\hspace{-0.25em}\cA^{(2)}_X\otimes \iota(\fM)\in \tilde \cH$) is the right (respectively, left) monoidal dual of $\cA^{(2)}_X \otimes \fM\in \tilde \cH$.

Finally, for the Calabi-Yau  structure, one can directly repeat the proof of 
Theorem~\ref{CY Hecke} with the functors $q_*$ and $p^!$ replaced by  their respective  monodromic
versions $\tilde q_*$ and $\tilde p^!$ as in the construction of the monoidal structure on $ \tilde \cH$.
Namely, they are given by the ringed-space variants where  $\tilde q_*$ is the composition of the usual $*$-pushforward
and then forgetting the module structure, and $\tilde p^!$ is the composition of the usual $!$-pullback
and then tensoring up with the free module.
\end{proof}



\subsection{Monodromic Hochschild calculations}
We continue with the setup and constructions of the previous section, in particular Theorem~\ref{monodromic Hecke}.
Thus suppose $X$ is a  stack with smooth diagonal, $Y$ is a stack, and $X\to Y$ is a proper representable map.
Suppose  
$\cA_X\in \D(X)$ is a commutative  algebra object whose underlying object is compact and hence dualizable
with dual denoted by $\cA_X^\vee\in \D(X)$.

Return to the fundamental correspondence
$$
\xymatrix{
 \cL Y  = Y\ti_{Y\ti Y} Y & \cL Y \ti_Y X =X \ti_{X \ti Y} X \ar[l]_-{q} \ar[r]^-{\epsilon} & X\ti_Y X.\\
}
$$
where $q$ is a base change of $p$, and $\epsilon$ is a base change $\delta$.

Define the  functor
$$
\xymatrix{
\tilde F: \D_{\cA^{(2)}_X}(X\ti_Y X) \ar[r]^-{\tilde \epsilon^!} & \D_{\cA_X}(\cL Y \ti_Y X) 
 \ar[r]^-{\tilde q_*}& \D(\cL Y) \\
}
$$
where $\tilde \epsilon^!$ is the composition of the pullback $\epsilon^!$ and then tensoring  with the diagonal $\cA_X$-module, and $\tilde q_*$ is the composition of forgetting $\cA_{X}$-module structure and then the pushforward $q_*$.

For simplicity, we will also assume given an equivalence of $\cA_X$-modules of the form $\cA_X^\vee \simeq \cA_X[k]$,
for some fixed $k$. This provides an identification of the right adjoint of the pushforward $\tilde q_*$ with a shift of the pullback 
$\tilde q^!$.
Thus we have the explicit continuous right adjoint
$$
\xymatrix{
\tilde F^r:\D(\cL Y)   \ar[r]^-{\tilde q^![k]}&\D_{\cA_X}(\cL Y \ti_Y X)  \ar[rr]^-{\tilde \epsilon_*[-2\dim X]} && \D_{\cA^{(2)}_X}(X\ti_Y X)\\
}
$$ 
where  $\tilde q^!$ is the composition of the pullback $q^!$ and then tensoring with $\cA_X$,  and
$\tilde \epsilon_*$ is the composition of forgetting to an $\cA^{(2)}_X$-module structure and then the pushforward $\epsilon_*$.

Let $\cI(\tilde F) \subset \D(\cL Y)$ denote the cocompletion of the  image of $\tilde F$.
It is naturally equivalent to the left orthogonal ${}^\perp \cK(\tilde F^r) \subset \D(\cL Y)$ to the kernel of $\tilde F^r$.

Suppose in addition that tensoring with $\cA_X^\vee$ is conservative so that $\tilde q^!$ satisfies descent. Then both of the above categories  are naturally equivalent to the full subcategory of $ \D(\cL Y)$ of objects whose image under $\tilde q^!$  is in the cocompletion of the image of $\tilde \epsilon^!$.

Now repeating the arguments of Theorem~\ref{thm center} and Corollary~\ref{cor trace} 
gives the following extension.

\begin{theorem}
There are natural adjoint commutative diagrams
$$
\xymatrix{
&  \tilde \cH \ar[dl]_-\ftr \ar[d]<-1ex>_-{F}&\\
 \Tr(\tilde \cH) \ar[r]^-\sim & \ar[u]<-1ex>_-{\tilde F^r} {}^\perp \cK(\tilde F^r) \simeq \cI(\tilde F) \ar[r]^-\sim
 &  \ar[ul]_-\fz \Z(\tilde \cH)
   }
$$
 \end{theorem}

\begin{proof}
For notational convenience, set $\tilde \cA = \D_{\cA_{\XYX}}(X \times_Y X)$ and $\tilde \cB = \D_{\cA_X}(X)$,
so that the center is the totalization of the cosimplicial object 
$$
\tilde \cC^\bullet =  \Hom_{\tilde \cB \otimes \tilde\cB^{op}} (\tilde\cA^{\otimes_{\tilde\cB \bullet}},  \tilde\cA)
$$

Using Proposition~\ref{prop int trans adjoint} as in the proof of of Theorem~\ref{thm center}, we have a level-wise fully faithful map
 of cosimplicial diagrams
 \[ 
\xymatrix{
\tilde \cC^\bullet \ar[r] &  
 \D_{\cA^{(\bullet + 1)}_{X}}(X^{\times_Y (\bullet+1)} \times_Y \cL Y)
 }
\]
where  the cosimplicial  structure maps on the latter are given by the composition of $!$-pullbacks together
with tensoring up with tensor powers of $\cA_X^\vee$.

Thus the center $\Z(\cA)$ is equivalent to the full subcategory
 of $\D(\cL Y)$ of objects whose image under $\tilde q^!$  is in the cocompletion of the image of $\tilde \epsilon^!$.
\end{proof}

Finally, there is a  Morita-invariance between the original and now monodromic calculations.

\begin{prop} The full subcategories $\cK(F^r), \cK(\tilde  F^r)\subset  \D(\cL Y)$ defined by the original
and monodromic fundamental correspondences coincide.
\end{prop}

\begin{proof}
On underlying objects, we calculate via composition identities and the projection formula  that
$
\tilde F^r(\fM) \simeq F(\fM) \otimes \cA'_{\XYX}$.
Since tensoring with $\cA_X^\vee$ is conservative, so is  tensoring with $\cA'_{\XYX}$, and hence
$\cK(F^r) = \cK(\tilde F^r) \subset \D(\cL Y)$.
\end{proof}


\subsection{Character sheaves again}

Now we  apply our previous results
to our motivating example when $X=BB$, $Y=BG$, for a reductive group $G$  and   Borel subgroup $B\subset G$. 
We take $\tilde X = BN$, for the unipotent radical $N\subset B$,
and take $\cA_X = \pi_*\CC_{BN} \in \D(BB)$ where $\pi:BN\to BB$ is the natural $H= B/N$ torsor.
Thus we are studying the bimonodromic Hecke category $\tilde \cH_G$ of unipotent bimonodromic $\D$-modules
on $N\bs G/N$.

\begin{theorem}\label{thm mon char shvs}
The category of unipotent character sheaves $\chsh_{G}$
is equivalent to both the trace and center 
of the bimonodromic Hecke category $\tilde \cH_{G}$.
The equivalences fit into natural adjoint commutative diagrams
$$
\xymatrix{
&  \tilde\cH_{G} \ar[dl]_-\ftr \ar[d]<-1ex>_-{\tilde F}&\\
 \Tr(\tilde\cH_{G}) \ar[r]^-\sim & \ar[u]<-1ex>_-{\tilde F^r} \chsh_{G} \ar[r]^-\sim
 &  \ar[ul]_-\fz \Z(\tilde\cH_G)
   }
$$
 \end{theorem}




\end{document}